\documentclass[11pt]{amsart}
\usepackage{latexsym}
\usepackage{amssymb,amsmath,mathabx, amscd}
\usepackage[pdftex]{graphicx}
\usepackage{enumerate}
\usepackage{colonequals, stmaryrd}
\usepackage{tikz}
\usepackage{tikz-cd}
\usepackage[margin=.9in]{geometry}

\usepackage{hyperref, mathrsfs}

\newcommand{\cc}{\mathbb{C}}
\renewcommand{\O}{\mathcal{O}}

\newcommand{\pp}{\mathbb{P}}
\newcommand{\hh}{\mathbb{H}}
\newcommand{\qq}{\mathbb{Q}}
\newcommand{\zz}{\mathbb{Z}}
\newcommand{\ff}{\mathbb{F}}
\newcommand{\kk}{\mathbb{K}}
\newcommand{\fE}{\mathcal{E}}
\newcommand{\fP}{\mathcal{P}}

\newcommand{\C}{\mathscr{C}}

\newcommand{\GL}{\text{GL}}
\newcommand{\SL}{\text{SL}}	 	
\newcommand{\PGL}{\text{PGL}}
\newcommand{\Stab}{\operatorname{Stab}}

\newcommand{\disc}{\operatorname{disc}}
\newcommand{\exc}{\operatorname{exc}}
\newcommand{\genus}{\operatorname{genus}}
\newcommand{\Frac}{\operatorname{Frac}}
\newcommand{\Spec}{\operatorname{Spec}}

\newcommand{\p}{\mathfrak{p}}

\renewcommand{\Im}{\operatorname{Im}}
\newcommand{\Ast}{\Asterisk}
\newcommand{\tr}{\text{tr}}

\newcommand{\Gal}{\text{Gal}}
\newcommand{\Frob}{\text{Frob}}
\newcommand{\Mat}{\operatorname{Mat}}
\newcommand{\Aut}{\operatorname{Aut}}

\newcommand{\tors}{\text{tors}}

\newcommand{\bs}{\operatorname{LGP}}

\newcommand{\gabold}{ \begin{pmatrix} a & \frac{bd_F(1-d_F)}{4} \\ b & a+bd_F \end{pmatrix}}
\newcommand{\habold}{ \begin{pmatrix} a & ad_F- \frac{bd_F(1-d_F)}{4} \\ b & -a \end{pmatrix}}

\newcommand{\sm}[4]{{\left(\begin{smallmatrix}#1&#2\\ #3&#4\end{smallmatrix}\right)}}
\newcommand{\sv}[2]{{\left(\begin{smallmatrix}#1 \\#2\end{smallmatrix}\right)}}

\newcommand{\vect}[2]{\begin{pmatrix} #1 \\ #2 \end{pmatrix}}

\makeatletter
\def\legendre@dash#1#2{\hb@xt@#1{%
  \kern-#2\p@
  \cleaders\hbox{\kern.5\p@
    \vrule\@height.2\p@\@depth.2\p@\@width\p@
    \kern.5\p@}\hfil
  \kern-#2\p@
  }}
\def\@legendre#1#2#3#4#5{\mathopen{}\left(
  \sbox\z@{$\genfrac{}{}{0pt}{#1}{#3#4}{#3#5}$}%
  \dimen@=\wd\z@
  \kern-\p@\vcenter{\box0}\kern-\dimen@\vcenter{\legendre@dash\dimen@{#2}}\kern-\p@
  \right)\mathclose{}}
\newcommand\legendre[2]{\mathchoice
  {\@legendre{0}{1}{}{#1}{#2}}
  {\@legendre{1}{.5}{\vphantom{1}}{#1}{#2}}
  {\@legendre{2}{0}{\vphantom{1}}{#1}{#2}}
  {\@legendre{3}{0}{\vphantom{1}}{#1}{#2}}
}
\def\dlegendre{\@legendre{0}{1}{}}
\def\tlegendre{\@legendre{1}{0.5}{\vphantom{1}}}
\makeatother

\newtheorem{thm}{Theorem}[section]
\newtheorem{ithm}{Theorem}
\newtheorem{iprop}[ithm]{Proposition}
\newtheorem{lem}[thm]{Lemma}

\newtheorem{prop}[thm]{Proposition}
\newtheorem{cor}[thm]{Corollary}

\theoremstyle{definition}
\newtheorem{defin}[thm]{Definition}

\newtheorem{question}[thm]{Question}

\theoremstyle{remark}
\newtheorem*{rem}{Remark}

\newcommand{\defi}[1]{\textsf{#1}} 

\title[A Local-global principle for isogenies]{A local-global principle for isogenies of composite degree} % This is the full title of the paper
\author{Isabel Vogt}

\thanks{This research was supported in part by the National Science Foundation Graduate Research Fellowship Program under Grant DGE-1122374 as well as Grant DMS-1601946.}

\begin{document}
\maketitle

\begin{abstract}
Let $E$ be an elliptic curve defined over a number field $K$.  
If for almost all primes of $K$, the reduction of $E$ has a rational cyclic isogeny of fixed
degree, then we can ask whether $E$ has a cyclic isogeny over $K$ of that degree.
Building upon the works of Sutherland, Anni, and Banwait-Cremona in the case of prime degree, we consider this question for cyclic isogenies of arbitrary degree.
\end{abstract}

\section{Introduction}

Fix a number field $K$.  Let $E$ be an elliptic curve defined over $K$.  For all primes $\p$ of $K$ where $E$ has good reduction, let $E_\p$ denote the reduction of $E$ modulo $\p$.  By assumption this is an elliptic curve over the residue field, which we denote by $k_{\p}$.  If $E$ has some level structure over $K$, such as a rational torsion point or a rational isogeny, then for almost all $\p$, the reduction $E_\p$ does as well.  We will say that $E$ has level-structure \defi{locally at $\p$} if $E_\p$ has such structure.

One can then ask about a converse: if $E$ has some structure locally at $\p$ for almost all $\p$, does $E$ necessarily have such structure over $K$?  Katz originally asked this question for the property that $m \mid \#E(K)_\tors$, and therefore, about rational $\ell$-torsion points, where $\ell$ is prime.  He showed that it is not true in general; however, $E$ is always isogenous (over $K$) to a curve $E'$ such that $m \mid \#E'(K)_\tors$ \cite[Theorem 2]{katz}.

Sutherland asked the analogous question for the property of having an isogeny of degree $\ell$, for a fixed prime $\ell$.  In \cite{sutherland}, he showed that this question again has a negative answer in general.  This local-global question cannot be salvaged by considering isogenous curves since the property of having an isogeny of prime degree $\ell$ is itself an isogeny invariant.  However, Sutherland gives a classification of the exceptions that implies, in particular, that if $\ell \equiv 1 \pmod{4}$ and $\sqrt{\ell} \not\in K$, then $E$ \emph{does} have an isogeny of degree $\ell$ over $K$.  Anni in \cite{anni} then proved that for any fixed number field $K$, there are only finitely many primes $\ell$ such that there exists an elliptic curve $E/K$ with an isogeny of degree $\ell$ locally almost everywhere, but not an isogeny of degree $\ell$ over $K$.  And if $\ell \neq 5,7$, there are finitely many $j$-invariants of curves over $K$ which give exceptions.

The key to the proof of these theorems is to translate the problem into a purely group-theoretic statement about the image of the mod $\ell$ Galois representation attached to the
$\ell$-torsion of a curve $E/K$ for which the local-global question is considered,
\[ \rho_{E,\ell} \colon G_K \rightarrow \Aut(E[\ell]) \simeq \GL_2(\zz/\ell\zz), \]
where $G_K$ denotes the absolute Galois group $\Gal(\bar{K}/K)$.
There exists a $K$-rational isogeny of degree $\ell$ if and only if $\Im(\rho_{E,\ell})$ preserves a $1$-dimensional subspace.
The Chebotarev density theorem shows that $E_\p$ admits a $k_\p$-rational isogeny of degree $\ell$ for almost all $\p$ if and only if \emph{every element }of $\Im(\rho_{E,\ell})$ preserves a $1$-dimensional subspace.

In this paper we extend these results to cyclic isogenies of composite degree $N$, which we will refer to as \defi{$N$-isogenies} for the remainder of the paper.  While the property of having an $\ell$-isogeny is an isogeny invariant, this is not true for isogenies of composite degree.  For this reason, following Katz, we will focus on ($K$-rational) isogeny classes of elliptic curves.  

For any field $k$, and $E/k$ an elliptic curve, we will denote by $\C(E/k)$ (or simply $\C(E)$ when the ground field is implicit) the $k$-rational isogeny class of $E$.  We say that $\C(E/k)$ \defi{has an $N$-isogeny} if there exist $E_1, E_2 \in \C(E/k)$ and a cyclic isogeny $E_1 \to E_2$ of degree $N$ defined over $k$.  For a number field $K$, we say that $\C(E/K)$ has an $N$-isogeny locally almost everywhere if $\C(E_\p/k_\p)$ has a ($k_{\p}$-rational) $N$-isogeny for almost all $\p$.

\begin{question}
If $\C(E/K)$ has an $N$-isogeny locally almost everywhere, must $\C(E/K)$ have an $N$-isogeny as well?
\end{question}

When this is true, we say that $\C(E/K)$ satisfies the local-global principle for $N$-isogenies (satisfies $\bs_N$ for short).  
If not, we say that $\C(E/K)$ is an \defi{exceptional isogeny class}.
One checks that for $j \neq 0,1728$, this depends only upon the $j$-invariant of $E$.
Define the set of isomorphism classes of exceptions
\[\Sigma(K, N) \colonequals \{ j \in K : \text{$j = j(E/K)$ and $\C(E)$ fails $\bs_N$} \}. \]
 If $j \in \Sigma(K, N)$, we say that $(N, j)$ is a \defi{exceptional pair} over $K$.  By the work of Sutherland, Anni, and Banwait-Cremona, we know that there are exceptions for some $N$ and $K$.  Our first  theorem is the following finiteness statement for counterexamples, see Section \ref{jinvs}:

\begin{ithm}\label{mainthm}
For any number field $K$,
\begin{enumerate}[1.]
\item\label{mainthm_1} If $N \notin \{ 5, 7, 8, 10, 24, 25, 32, 40, 49, 50, 72\}$, then the set $\Sigma(K, N)$ of $j$-invariants of exceptions to $\bs_N$
is finite.
\item\label{mainthm_2} The union $ \Lambda(K) \colonequals \displaystyle\bigcup_{\gcd(N, 70) = 1} \Sigma(K,N)$ is finite.
\end{enumerate}
\end{ithm}

If $N = \prod_i \ell_i^{n_i}$, then $\C(E)$ has an $N$-isogeny locally almost everywhere if and only if for all $i$, $\C(E)$ has an $\ell_i^{n_i}$-isogeny locally almost everywhere.  And $\C(E)$ has a global $N$-isogeny if and only if $\C(E)$ has a global $\ell_i^{n_i}$-isogeny for all $i$.  

For this reason, we begin by classifying exceptions when $N = \ell^n$ is a prime power.  Given the results of \cite{sutherland}, \cite{anni}, and \cite{bc} in the case $n=1$, a key component of the proof of Theorem \ref{mainthm} is understanding when a global $\ell$-isogeny can be ``lifted" to a global $\ell^n$-isogeny using the local data of an $\ell^n$-isogeny locally almost everywhere.  We make the following definition:

\begin{defin}
Assume that $\C(E)$ globally has an $\ell$-isogeny, and locally almost everywhere has an $\ell^n$-isogeny.  If $\C(E)$ does not globally have an $\ell^n$-isogeny, then we say that $\ell$ is a \defi{lift-exceptional prime} for $K$ and call $(\ell^n, j(E))$ a \defi{lift-exceptional pair} for $K$.
\end{defin}

\begin{ithm}\label{ellnthm}
Let $K$ be a number field and $\ell$ an odd prime.  If $\ell$ is lift-exceptional, then $\ell \leq 6[K:\qq] +1$.
Moreover, for fixed $K$, there are finitely many $j$-invariants $j(E) \in K$ such that there exists a prime power $\ell^n$ such that $(\ell^n, j(E))$ is exceptional.
\end{ithm}

\begin{rem}
This result bounds lift-exceptional primes by a constant depending only upon $K$, not $E$.  This constant agrees with the bound for prime degree isogenies found by Anni \cite{anni} when $\sqrt{\left(\frac{-1}{\ell}\right) \ell} \notin K$.  It is not possible to strengthen this to an absolute constant.  For any $\ell$, there exists a number field $K$ and an elliptic curve $E$ defined over $K$ such that $\ell$ is a lift-exceptional prime (see Corollary \ref{counter}).
\end{rem}

The key computation underpinning Theorem \ref{ellnthm} is the following result (Theorem \ref{upthm}, which we prove in Sections \ref{lemup}  and \ref{exceptional}), which implies that lift-exceptional curves have highly constrained Galois action on their $\ell$-power torsion points.

\begin{iprop}\label{ellnthm_group}
Let $E$ be an elliptic curve over a number field $K$, and let $\ell$ be an odd prime such that $\C(E)$ has a global $\ell$-isogeny and locally almost everywhere has an $\ell^n$-isogeny.
If $\ell^k$ is the smallest power of $\ell$ for which $\C(E)$ fails to have an $\ell^k$-isogeny, then $k$ is odd.  Writing $k=2m+1$, for some $E'$ in the isogeny class of $E$, the image of $\rho_{E', \ell^{k}}$ must, up to conjugacy, be contained in the group
\[ R(\ell^{2m+1}) = \left\{ \begin{pmatrix} r & s \\ \ell^{2m} (\epsilon s) & \epsilon t \end{pmatrix} \ : \ r\equiv t \pmod{\ell^{m+1}}, \epsilon = \pm 1 \right\}. \]
Furthermore, $\C(E)$ has an $\ell^k$-isogeny defined over some quadratic extension of $K$.
\end{iprop}

Proposition \ref{ellnthm_group} applies only to odd primes $\ell$.  In Proposition \ref{2exceptions} we classify all exceptional Galois representations for $2^n$-torsion for $n \leq 6$ (in particular minimal ones).  This covers all counterexamples to the local-global principle for $2^n$-isogenies that could occur infinitely often over a number field.

When the base change $E_{\bar{K}}$ of $E$ to $\bar{K}$ has the extra structure of complex multiplication (\defi{geometric complex multiplication}) by an order $\O$ in an imaginary quadratic field $F$, we can classify exceptional pairs up to a factor which is polynomially small in the degree of $K$ over $\qq$.  We denote the Hilbert class field of $F$ by $H_F$.

\begin{ithm}\label{main_cm}
Let $E/K$ be an elliptic curve with geometric complex multiplication by an order $\O \subset \O_F$.  If $C$ is any integer of the form $C = \prod_i \ell_i^{n_i}$ where for all $i$, $\ell_i$ splits in $F$ and one of
\begin{itemize}
\item $\ell_i \equiv 1 \pmod{4}$ and $K \supset \qq(\sqrt{\ell_i})$, or
\item $\ell_i \equiv 3 \pmod{4}$ and $KF = K(\sqrt{-\ell_i})$, or
\item $\ell_i = 2$, $n_i \geq 3$ and $K \supset \qq(\sqrt{2})$ and $KF = K(\sqrt{-2})$,
\end{itemize}
then $\C(E)$ has a $C$-isogeny locally almost everywhere.  Conversely, if $\C(E)$ has an $N$-isogeny locally almost everywhere, then there exist relatively prime numbers $A, B, C$ such that $N = ABC$ and
\[A \leq (\#\O_F^\times [KF : H_F])^4 \leq (6[K:\qq])^4, \]
$\C(E)$ has a $B$-isogeny,  $\C(E)$ fails to have a $C$-isogeny, and $C$ is of the form above.  Moreover, if $F \subset K$ then $C=1$.
\end{ithm}

When $N = \ell$ is prime, the quartic polynomial bound for $A$ can be improved to linear.  More precisely, we show in Section \ref{cm} that for $\ell > 6[K:\qq]+1$, the pair $(\ell, j(E))$ is exceptional if and only if $F \not \subset K$ and $\ell$ is of the form of the integer $C$ in the statement of Theorem \ref{main_cm}.  This shows that if $K$ does not contain the CM field $F$, then exceptional primes can be arbitrarily large compared to the degree $[K:\qq]$, correcting \cite[Lemma 6.1]{anni} and the final remark in \cite{sutherland}.

In addition to classifying exceptions for general number fields $K$, Sutherland proves that there is exactly one counterexample to the local-global principle for $\ell$-isogenies when $K=\qq$, namely $(\ell, j(E)) = (7, 2268945/128)$.  We extend this to prime power degrees and prove

\begin{ithm}\label{rationalpoints}
Let \(\C(E/\qq)\) have an \(\ell^n\)-isogeny locally almost everywhere.  If \(E\) is not \(\qq\)-isogenous to a curve with an \(\ell^n\)-isogeny over \(\qq\), then \(\ell = 7\), \(n=1\) or \(2\), and \(j(E) = 2268945/128\).
\end{ithm}
%Whereas in the case of prime degree isogenies there were only finitely many counterexamples over $\qq$, part (2) of Theorem \ref{rationalpoints} gives an infinite family with exceptions for isogenies with degree a power of $2$. 

As for isogenies of prime degree, these theorems are proved by analyzing the mod $N$ Galois representation attached to $E$ and modular curves $X_H$ for exceptional subgroups $H \subseteq \GL_2(\zz/N\zz)$.  
The paper is organized as follows: results on prime power isogenies and lift-exceptional pairs are covered in Sections \ref{sec_gpthy} -- \ref{jinvs} of the paper.  In Section \ref{prelim} we cover the basic preliminaries on Galois representations and modular curves, to set the stage for the remainder of the paper.  In Section \ref{sec:rephrase} we rephrase the local-global question in terms of $\Im(\rho_{E, \ell^n})$ as a subgroup of $\GL_2(\zz/\ell^n \zz)$ and lay out a framework for the group-theoretic analysis to follow in Section \ref{sec_gpthy}.  The group theory necessary to classify lift-exceptional subgroups of $\GL_2(\zz/\ell^n \zz)$ is contained in Sections \ref{lemup} and \ref{exceptional}.  At this point we will be able to deduce Proposition \ref{ellnthm_group}.  The results of previous sections do not apply to the prime $2$, so in Section \ref{two} we summarize the techniques used to computationally investigate this problem.  In Section \ref{boundell} we prove the boundedness of lift-exceptional primes stated in Theorem \ref{ellnthm}.   We prove the finiteness of lift-exceptional pairs from Theorem \ref{ellnthm} and the finiteness result of Theorem \ref{mainthm} in Section \ref{jinvs} by examining the appropriate modular curves.  Section \ref{cm} contains more precise characterizations of exceptional primes for curves with complex multiplication and a proof of Theorem \ref{main_cm}.  Finally, in Section \ref{qpoints}, we turn to the problem of finding the rational points on the relevant modular curves to prove Theorem \ref{rationalpoints}.  This relies heavily on the recent work on Rouse--Zureick-Brown \cite{rzb} and Sutherland--Zywina \cite{sz}.

%\tableofcontents

\section{Preliminaries}\label{prelim}

\subsection{Galois Representations}

Let $E$ be an elliptic curve over $K$ and $N$ a natural number.  Fix an algebraic closure $\bar{K}$ of $K$.  We will denote by $E[N]$ the group of $N$-torsion points of $E$ over $\bar{K}$, which is endowed with a linear action of $\Gal(\bar{K}/K) \equalscolon G_K$.  We therefore have a map $G_K \to \Aut(E[N])$.
Choosing an isomorphism $(\zz/N\zz)^2 \simeq E[N]$ gives an identification of $\Aut(E[N])$ with $\GL_2(\zz/N\zz)$ by action on the left (via column vectors).  This defines a representation
\[\rho_{E,N} \colon G_K \to \GL_2(\zz/N\zz), \]
which we refer to as the \defi{mod $N$ Galois representation} attached to $E/K$.
Choosing compatible bases for all $E[N]$, 
we can define the \defi{adelic Galois representation}
\[\rho_{E, \infty} \colon G_K \to \Aut\left( \varprojlim_{N} E[N]\right) \simeq \GL_2(\hat{\zz}). \]
Composition with reduction modulo $N$ recovers the mod $N$ Galois representation $\rho_{E, N}$.
Similarly composition with the surjection $\hat{\zz} \to \zz_\ell$ defines the \defi{$\ell$-adic Galois representation} 
\[ \rho_{E,\ell^\infty} \colon G_K \rightarrow \GL_2(\zz_\ell), \]
which is the inverse limit of the mod $\ell^n$ Galois representations.
Let us fix the notation $G_E \subset \GL_2(\hat{\zz})$ for the image of the adelic Galois representation of an elliptic curve $E/K$.  Let $G_E(\ell^n)$ denote the image of $G_E$ under reduction $\GL_2(\hat{\zz}) \to \GL_2(\zz/\ell^n \zz)$, i.e.\ the image of the mod $\ell^n$ representation.

\subsection{Subgroups of $\GL_2(\zz/\ell^n \zz)$}

Recall that a \defi{Borel subgroup} of $\GL_2(\zz/\ell^n \zz)$ is the subgroup of automorphisms of $(\zz/\ell^n \zz)^2$ that preserve a specified submodule $L \simeq \zz/\ell^n \zz \subset (\zz/\ell^n \zz)^2$ (which we will refer to as a line).  Choosing a basis compatible with the line $L \subset \left( \zz/\ell^n\zz\right)^2$, the Borel subgroup associated to $L$ can be identified with matrices of the form $\sm{\Ast}{\Ast}{0}{\Ast}$.

Call two lines $L$ and $L'$ in $\left( \zz/\ell^n \zz\right)^2$ \defi{independent} if they give a direct sum decomposition $L \oplus L' \simeq \left( \zz/\ell^n \zz\right)^2$.  Associated to two such lines is a \defi{split Cartan subgroup} consisting of linear automorphisms of $ \left( \zz/\ell^n \zz\right)^2$ preserving both $L$ and $L'$.  Again in an appropriate basis, a split Cartan subgroup can be identified with the subgroup of diagonal matrices, i.e., matrices of the form $\sm{\Ast}{0}{0}{\Ast}$.

We make one new definition which will be useful in what follows:

\begin{defin}
A \defi{radical subgroup} of $\GL_2(\zz/\ell^n\zz)$ is one that fixes a line $L \subset (\zz/\ell^n\zz)^2$ and an isomorphism between $L$ and the quotient $(\zz/\ell^n\zz)^2/L$ up to a sign.  Hence there exists a basis in which a radical subgroup acts as
\[\begin{pmatrix} \chi & \Ast \\ 0 & \pm \chi \end{pmatrix}.\]
\end{defin}
Radical subgroups occur ``in nature" as the image of $\rho_{E, \ell}$ for $E$ with CM by an order in a field $F$ in which $\ell$ ramifies \cite[Thm 13.1.2]{gross}.

Finally, recall that a \defi{nonsplit Cartan subgroup} of $\GL_2(\zz/\ell\zz)$ is a cyclic subgroup isomorphic to $\ff_{\ell^2}^\times$ acting on the $\ff_\ell$-vector space $\ff_{\ell^2} \simeq \ff_{\ell}^2$.  Explicitly, when $\ell \neq 2$, let $\epsilon$ be a nonquadratic residue mod $\ell$.  Then in an appropriate basis, a nonsplit Cartan subgroup can be identified with matrices of the form $\sm{x}{\epsilon y}{y}{x}$  for $x,y \in \ff_\ell$ not both $0$.

There is an involution on the set of subgroups of $\GL_2(\zz/N\zz)$ sending a group to its transpose
\[G^T \colonequals \{g^T : g \in G\}.\]
Many relevant subgroups of $\GL_2(\zz/N\zz)$ are conjugate to their transpose; for example Borel and Cartan subgroups.

\subsection{Modular Curves}

We begin by recalling the definition of the modular curve $X(N)$.  This is the coarse space of the smooth compactification of the stack whose $S$ points parameterize (up to isomorphism) pairs $(E, \iota)$, where $E/S$ is an elliptic curve and 
\[\iota \colon \left(\zz/N\zz\right)^2_S \to E[N] \]
is an isomorphism.  Note that this is the ``big" modular curve at level $N$: it is geometrically disconnected, with components over $\qq(\mu_N)$ in bijection with the primitive elements of $\mu_N$.

Precomposition of $\iota$ with $g^{-1} \in \GL_2(\zz/N\zz)$ 
\[ g(\iota) = \iota \circ g^{-1} \]
defines a \emph{left} action of $\GL_2(\zz/N\zz)$ on $X(N)$.
The map $X(N)$ to the $j$-line $X(1)$ forgetting the level structure at $N$ is a Galois covering with group $\GL_2(\zz/N\zz)/\{\pm 1\}$. 
Note that the Galois group $G_\qq$ also acts on $X(N)$ by postcomposition with $\iota$ on the left.

Let $G$ be a subgroup of $\GL_2(\zz/N\zz)$ containing $-I \colonequals \sm{-1}{0}{0}{-1}$.  We can then define the modular curve $X_G$ as the (coarse space of the) quotient of $X(N)$ by the action of $G \subseteq \GL_2(\zz/N\zz)$.  The $K$-points of $X_G$ parameterize pairs $(E, \mathcal{C})$, where $E$ is an elliptic curve over $K$ and $\mathcal{C}$ is an equivalence class of isomorphisms $\iota \colon (\zz/N\zz)^2 \to E[N]$, where $\iota \sim \iota'$ if there exists $g \in G$ such that $\iota' = \iota \circ g^{-1}$.

\begin{lem}[{\cite[Lemma 2.1]{rzb}}]
There exists $\iota$ such that $(E, \iota)$ gives rise to a $K$-point of $X_G$ if and only if $\Im(\rho_{E, N})$ is contained in a subgroup of $\GL_2(\zz/N\zz)$ conjugate to $G$.
\end{lem}

Note that in \cite{rzb} the Galois representation is defined in terms of row vectors, and so one must take the transpose $G \mapsto G^T$ to match our modular curves.

If $G$ does not have surjective determinant, then $X_G$ is also geometrically disconnected.  It is classical that a connected component $X(N)^\circ$ of $X(N)$ can be described geometrically by
\[X(N)^\circ(\cc) \simeq \Gamma(N) \backslash \hh^*, \]
where $\hh^*$ is the extended upper half plane $\hh \cup \pp^1(\qq)$ and $\Gamma(N)$ is the full modular group of level $N$.  There is a similar description in the case of $X_G$.  Given a subgroup $S \subset \SL_2(\zz/N\zz)$ define the congruence subgroup 
\[\Gamma_S \colonequals \{s\in \SL_2(\zz) : (s \mod N) \in S \}. \]
For convenience let $\bar{G} \colonequals G\cap \SL_2(\zz/N\zz)$.

\begin{lem}
Let $G \subset \GL_2(\zz/N\zz)$ contain $-I$.  Then the connected components of $X_G$ are indexed by $G/\bar{G}$ and the complex points of one connected component $X_G^\circ$ can be realized as
\[ X_G^\circ(\cc) \simeq \Gamma_{\bar{G}} \backslash\hh^*. \]
\end{lem}
\begin{proof}
This follows from the classical fact for $X(N)$.  Indeed by the description of the components of $X(N)(\cc)$, the action of $\bar{G}$ preserves the components (and is in fact the stabilizer of any one component).  This action restricts on $\Gamma(N) \backslash \hh^*$ to the action of $\Gamma_{\bar{G}} / \Gamma(N) \simeq \bar{G}$.

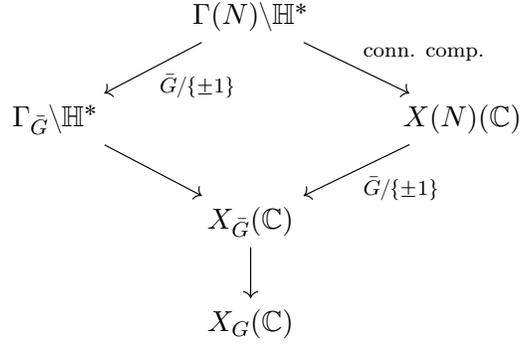
\begin{center}
\begin{tikzcd}
&  \Gamma(N) \backslash\hh^* \arrow{ld}{\bar{G}/\{\pm 1\}} \arrow{rd}{\text{conn. comp.}} \\
\Gamma_{\bar{G}} \backslash \hh^*\arrow{rd} & & X(N)(\cc) \arrow{ld}{\bar{G}/\{\pm 1\}} \\
& X_{\bar{G}}(\cc) \arrow{d} \\
& X_G(\cc) 
\end{tikzcd}
\end{center}

The result now follows from the following Lemma.
\end{proof}

\begin{lem}
Let $C = \coprod_{i \in I} C_i$ be a smooth (possibly disconnected) curve with componets indexed by $i \in I$.  Let $G$ be a finite group acting on $C$.  In particular, $G$ acts on $I$ and we write $\Stab_G(i)$ for the stabilizer in $G$ of $i \in I$.  Then
\[C/G \simeq \coprod_{i \in I/G} C_i / \Stab_G(i).\]
\end{lem}
\begin{proof}
Write $G \cdot i$ for the orbit of $i \in I$ under $G$.  Clearly we have
\[C \simeq \coprod_{i \in I/G} \left( \coprod_{j \in G\cdot i} C_j\right).\]
Since the action of $G$ respects this decomposition into orbits, it suffices to consider the case where there is only one orbit.  The result now follows from the fact that if $G$ acts transitively on $I$, then $C/G \simeq C_i / \Stab_G(i)$
for any $i \in I$.
\end{proof}

In particular this implies that we may compute the genus of $X_G$ using the description $\Gamma_{\bar{G}} \backslash\hh^*$.

\section{Group-theoretic rephrasing}\label{sec:rephrase}

We now translate the conditions of $\C(E)$ having an $\ell^n$-isogeny and $\C(E)$ locally almost everywhere having an $\ell^n$-isogeny into the language of Galois representations.

\subsection{Prime power isogenies in $\C(E)$}

We say cyclic isogenies $\phi_1$ and $\phi_2$ with kernels $C_1 \subset \left(\zz/N_1 \zz\right)^2$, $C_2 \subset \left(\zz/N_2 \zz\right)^2$ are \defi{independent} if the lines $\bar{C}_1, \bar{C}_2$ are independent in $\left(\zz/\gcd(N_1, N_2) \zz\right)^2$.

\begin{lem}\label{lem:isogeny_class_has_elln}
Suppose that \(E/K\) is \(\ell^a\)-isogenous to a curve \(E'/K\) which has independent \(\ell^b\)- and \(\ell^c\)-isogenies (which are also independent from the dual of the \(\ell^a\)-isogeny) with \(b+c = n\), and \(0 \leq a\leq c \leq b\).  Then, up to conjugacy, the mod \(\ell^n\) Galois representation \(G_E(\ell^n)\) is contained in the group
\[ A(\ell^n) = A_{a, c, b}(\ell^n) = \left\{ \begin{pmatrix} \alpha & \beta \\ \gamma & \delta \end{pmatrix} : \gamma \equiv 0 \pmod{\ell^{a+b}}, \delta \equiv \alpha - \ell^a \beta \pmod{\ell^c} \right\}.\]
\end{lem}
\begin{rem}
Notice that when \(a=0\), this recovers the case that \(G_{E'}(\ell^n)\) is contained in a Cartan subgroup modulo \(\ell^c\) and a Borel subgroup modulo \(\ell^b\) for \(b + c = n\).
\end{rem}

\begin{proof}[Proof Sketch of Lemma \ref{lem:isogeny_class_has_elln}]
Choose a basis \(e_1, e_2\) for \(\ell\)-adic Tate module \(T_\ell E'\) of \(E'\) so that, modulo \(\ell^b\) and \(\ell^c\), the vectors \(e_1, e_2\) correspond to the given independent \(\ell^b\)- and \(\ell^c\)-isogenies.  Since \(E'\) also has an independent \(\ell^a\)-isogeny (and \(a \leq c, b\)), the projectivization of the \(\ell^a\)-torsion is a trivial Galois module.  Hence,
\[ G_{E'}(\ell^n) \subset \left\{ \begin{pmatrix} x & \ell^c y \\ \ell^bz & x + \ell^aw  \end{pmatrix} : x,y,z,w \in \zz/\ell^n\zz \right\}.\]
By the \(\ell^a\)-isogeny, the \(\ell\)-adic Tate module \(T_\ell E\) of \(E\) injects into \(T_\ell E'\) as an index \(\ell^a\)-sublattice \(\zz_\ell v + \ell^a\zz_\ell^2\) for some \(v \in \zz_\ell^2\).  Since \(v, e_1, e_2\) are all distinct modulo \(\ell\), by acting by \(\PGL_2\), we may assume that \(v = e_1 + e_2\).  To obtain the Galois action on the lattice spanned by \(\ell^ae_1\) and \(e_1 + e_2\), we conjugate
\[ \ell^{-a}\begin{pmatrix} 1 & -1 \\ 0 & \ell^a \end{pmatrix}\begin{pmatrix} x & \ell^c y \\ \ell^bz & x + \ell^aw  \end{pmatrix} \begin{pmatrix} \ell^a & 1 \\ 0 & 1 \end{pmatrix} 
%=  \ell^{-a}\begin{pmatrix} 1 & -1 \\ 0 & \ell^a \end{pmatrix}\begin{pmatrix} \ell^ax & x + \ell^cy \\ \ell^{a+b}z & \ell^bz + x + \ell^aw \end{pmatrix} \]
%\[= \ell^{-a}\begin{pmatrix}\ell^a x - \ell^{a+b}z & \ell^cy - \ell^bz - \ell^a w \\ \ell^{2a + b}z & \ell^{a+b}z + \ell^ax + \ell^{2a} w  \end{pmatrix} 
= \begin{pmatrix} x - \ell^{b}z & \ell^{c-a}y - \ell^{b-a}z - w \\ \ell^{a + b}z & \ell^{b}z + x + \ell^{a} w  \end{pmatrix}.\]
The group of matrices of this shape is exactly \(A(\ell^n)\).
%To simply the presentation, since \(b \geq c\), we can replace \(y - \ell^{b-c}z\) with \(y\) and then replace \(-w + \ell^{c-a} y\) to obtain 
%\[\begin{pmatrix} x - \ell^{b}z & w \\ \ell^{a + b}z & \ell^{b}z + x - \ell^{a} w  +\ell^c y\end{pmatrix}.\]
%Then replace \(x - \ell^b z\) with \(x\) to obtain
% \[\begin{pmatrix} x  & w \\ \ell^{a + b}z &  x + 2\ell^bz - \ell^{a} w  +\ell^c y\end{pmatrix}.\]
% Since \(b \geq c\), we can again replace \(y + 2\ell^{b-c}z\) with \(y\) to obtain
%  \[\begin{pmatrix} x  & w \\ \ell^{a + b}z &  x  - \ell^{a} w  +\ell^c y\end{pmatrix}.\]
\end{proof}

An isogeny class \(\C\) has an \(\ell^n\)-isogeny if and only if for every element \(E \in \C\), we have that \(G_E(\ell^n)\) is conjugate to a subgroup of \(A_{a,c,b}(\ell^n)\) for some choice of \(0 \leq a \leq c \leq b\) with \(b + c = n\).

The local condition can be stated more simply.  For $j \leq n/2$, write $\Delta_j \subseteq \GL_2(\zz/\ell^n \zz)$ for the subgroup of matrices that are diagonal modulo $\ell^j$ and upper-triangular modulo $\ell^{n-j}$. 

\begin{lem}\label{lem_exceptional_criterion}
Let $M \in \Mat_2(\zz/\ell^n\zz)$ be a $2 \times 2$ matrix. Then the characteristic polynomial $\chi_M(t) = t^2 - \tr(M)t + \det(M)$ has a root in $\zz/\ell^n\zz$ if and only if, up to conjugation by an element of $\GL_2(\zz/\ell^n\zz)$, $M$ is diagonal modulo $\ell^j$ and upper-triangular modulo $\ell^{n-j}$ for some $j \leq n/2$.
\end{lem}

\begin{proof}
In the reverse direction, up to conjugacy we have that $M = \begin{pmatrix} a & \ell^j b \\ \ell^{n-j}c & d \end{pmatrix}$, and so $\chi_M(t) = (t-a)(t-d)$ has a root in $\zz/\ell^n\zz$.

For the forward direction, we first record the following elementary observations.  Let $I$ denote the identity matrix.  
\begin{enumerate}[(I)]
\item\label{shift} If $M = M' + cI$ for some constant $c \in \zz/\ell^n\zz$, then 
\[\chi_M(t) \text{ has a root mod $\ell^n$} \qquad \Leftrightarrow \qquad \chi_{M'}(t)\text{ has a root mod $\ell^n$}. \]
\item\label{invt_sc} If $M = aM''$ for some constant $a \in (\zz/\ell^n\zz)^\times$, then 
\[\chi_M(t) \text{ has a root mod $\ell^n$} \qquad \Leftrightarrow \qquad \chi_{M''}(t)\text{ has a root mod $\ell^n$}. \]
\item\label{ell_sc} If $M = \ell M'''$, then
\[\chi_M(t) \text{ has a root mod $\ell^n$} \qquad \Leftrightarrow \qquad \chi_{M'''}(t)\text{ has a root mod $\ell^{n-2}$}. \]
\end{enumerate}
\noindent We now break into the following cases:
\begin{enumerate}
\item \textit{$M$ has distinct eigenvalues modulo $\ell$.}  In this case, the projectivized eigenvector equation for $M$ has a solution mod $\ell$ with nonzero derivative.  Hence by Hensel's lemma it has a solution modulo all orders.  Therefore $M$ has an eigenvector in $(\zz/\ell^n \zz)^2$, and hence is conjugate to an upper-triangular matrix.
\item \textit{$M$ is scalar modulo $\ell$.}  Up to subtracting a scalar matrix and applying \eqref{shift}, we have that $M = \ell M'$; hence by \eqref{ell_sc} and induction we can assume that $n\leq 2$.  The conclusions are then satisfied as it is diagonal modulo $\ell$.
\item  \textit{$M$ is not diagonalizable modulo $\ell$.}  In this case, up to conjugating, adding a scalar \eqref{shift}, and scaling \eqref{invt_sc}, we can assume that $M = \begin{pmatrix} 0 & 1 \\ a\ell & b \ell \end{pmatrix}$.  In this case $\begin{pmatrix} 1 \\ x \end{pmatrix}$ is an eigenvector precisely when $x$ is a root of $\chi_M(t)$.
\end{enumerate}
\end{proof}

\subsection{Chebotarev density theorem}\label{cheb}

The Chebotarev density theorem will allow us to translate $\bs_{\ell^n}$ into a statement only involving the image of the global Galois representation attached to $E$ as a subgroup of $\GL_2(\zz/ \ell^n\zz)$.

Let $K_{n} \colonequals K(E[\ell^n])$ be the $\ell^n$-division field of $E$ and assume that $\p \nmid \ell^n \cdot N_E$.  Then the residue field extension of $ K_n/K$ for any prime above $\p$ is the $\ell^n$-division field of $E_\p$.  As the Galois group $\Gal(k_\p(E_\p[\ell^n])/k_\p)$ is cyclic, generated by the $|k_\p|$-power Frobenius $\varphi_\p \colonequals (x \mapsto x^{|k_\p|})$, the condition that $\rho_{E_\p, \ell^n}(G_{k_\p})$ is conjugate to a subgroup of $G$ reduces to the condition that $\rho_{E_\p, \ell^n}(\varphi_\p)$ is conjugate to a element of $G$.
Under our assumption on $\p$, the pair $(E, E[\ell^n])$ has good reduction at $\p$.  Explicitly, if $P_1, P_2$ are a basis of the $\ell^n$-torsion $E[\ell^n]$ used to define the mod $\ell^n$ Galois representation, we can choose an integral model $(\fE, (\fP_1, \fP_2))$ over $\Spec \O_{K}\left[\frac{1}{\ell^n\cdot N_E}\right]$ so that over the special fiber $E_{\p}$, the points $\bar{P_1}, \bar{P_2}$ are a basis of $E_{\p}[\ell^n]$.  
Hence the condition that $\rho_{E_\p, \ell^n}(\varphi_\p)$ is conjugate to an element of $G$ is equivalent to the condition that $\rho_{E, \ell^n}(\Frob_\p)$ is conjugate to an element of $G$, where $\Frob_\p$ is a Frobenius element at $\p$.

\begin{lem}\label{lem_localgroup}
The following are equivalent:
\begin{enumerate}[(a)]
\item\label{local_a} The isogeny class \(\C(E)\) has an \(\ell^n\)-isogeny locally almost everywhere;
\item\label{local_b} Every element of $\rho_{E, \ell^n}(G_K)$ has a root of its characteristic polynomial mod \(\ell^n\);
\item\label{local_c} Every element of $\rho_{E, \ell^n}(G_K)$ is conjugate to a element of one of $\left\{\Delta_0, \dots, \Delta_{\lfloor \frac{n}{2} \rfloor}\right\}$.
\end{enumerate}
\end{lem}
\begin{proof}
By Lemma \ref{lem_exceptional_criterion}, criteria \eqref{local_b} and \eqref{local_c} are equivalent.  We will show that \eqref{local_a} implies \eqref{local_b} and that \eqref{local_c} implies \eqref{local_a}.  

Suppose that \eqref{local_a} holds and let \(E'_\p \in \C(E_\p/k_\p)\) have an \(\ell^n\)-isogeny.  Then the generator $\rho_{E'_{\p}, \ell^n}(\varphi_{\p})$ has a root of its characteristic polynomial mod \(\ell^n\).  Since the (integral) trace and determinant of Frobenius are isogeny invariants, this implies that the generator $\rho_{E_{\p}, \ell^n}(\varphi_{\p})$ has a root of its characteristic polynomial mod \(\ell^n\).  By the local-global compatibility, this implies that $\rho_{E, \ell^n}(\Frob_\p)$ has a root of its characteristic polynomial mod \(\ell^n\).  By the Chebotarev density theorem, this implies \eqref{local_b}.

Now assume \eqref{local_c}.  By the Chebotarev density theorem, for every good prime $\p \nmid \ell^n  N_E$,  we have that $\rho_{E, \ell^n}(\Frob_\p)$ is conjugate to an element of one of $\left\{\Delta_0, \dots, \Delta_{\lfloor \frac{n}{2} \rfloor}\right\}$.  Hence \(E_\p\) has independent \(\ell^j\) and \(\ell^{n-j}\) isogenies for some \(0 \leq j \leq n/2\).  By composing with the dual of one of these isogenies, we see that \(\C(E_\p/k_\p)\) has an \(\ell^n\)-isogeny.
\end{proof}

\subsection{Exceptional subgroups}

In light of the Chebotarev density theorem and its implications for local and global level structure explicated in $\S$\ref{cheb}, whether or not $\C(E)$ satisfies the local-global principle for $\ell^n$-isogenies depends only upon $G_E(\ell^n)$ as a subgroup of $\GL_2(\zz/\ell^n\zz)$.

In light of this, we begin by making several definitions of exceptional subgroups.  If $G \subseteq \GL_2(\zz/\ell^n\zz)$, we use the notation $G(\ell^j)$ to denote the image of $G$ under the surjection $\GL_2(\zz/\ell^n\zz) \to \GL_2(\zz/\ell^j\zz)$.

\begin{defin}
A subgroup \(G \subset \GL_2(\zz/\ell^n \zz)\) is \defi{exceptional} if 
\begin{enumerate}[(i)]
\item for every element \(g \in G\), the characteristic polynomial \(\chi_g(t)\) has a root in \(\zz/\ell^n \zz\); equivalently, if \(\ell \neq 2\), for all \(g \in G\), \(\Delta(g) = \tr(g)^2 - 4 \det(g)\) is a square in \(\zz/\ell^n \zz\);
\item there do not exist \(0 \leq a \leq c \leq b\) with \(b+c = n\) such that \(G\) is conjugate to a subgroup of \(A_{a,c, b}(\ell^n)\).
\end{enumerate}
\end{defin}

Given that we work inductively modulo higher and higher powers of $\ell$,
it is convenient to have a group-theoretic phrasing of the condition that $\C(E)$ satisfies the local-global principle for $\ell^{n-1}$-isogenies but fails the local-global principle for $\ell^{n}$-isogenies.

\begin{defin}\label{lift_defi}
We will say that a subgroup \(G \subset \GL_2(\zz/\ell^n \zz)\) is \defi{lift-exceptional (at step \(n\))} if 
\begin{enumerate}[(i)]
\item\label{defin:lift_1} \(G(\ell^{n-1})\) is contained in a Borel subgroup of \(\GL_2(\zz/\ell^{n-1}\zz)\);
\item\label{defin:lift_2} for every \(g \in G\), the characteristic polynomial \(\chi_g(t)\) has a root in \(\zz/\ell^{n-1} \zz\);
\item\label{defin:lift_3} \(G(\ell)\) is not contained in a split Cartan subgroup of \(\GL_2(\zz/\ell\zz)\);
\item\label{defin:lift_4} \(G\) is not contained in a Borel subgroup of \(\GL_2(\zz/\ell^n\zz)\);
\item\label{defin:lift_5} there does not exist \(0 \leq a \leq \lfloor (n-1)/2 \rfloor\) such that \(G\) is conjugate to a subgroup of \(A_{a, a+1, n-a-1}(\ell^n)\).
\end{enumerate}
\end{defin}

\subsection{Overview of group-theoretic techniques for Proposition \ref{ellnthm_group}}
We will prove Proposition \ref{ellnthm_group} by inductively considering when we may perform a sequence of lifts of our global $\ell$-isogeny in $\C(E)$ to an $\ell^2$-isogeny, then an $\ell^3$-isogeny, and so forth to successively higher powers of $\ell$.  The technique applied splits the problem naturally into two cases: (1) we are lifting from an odd power of $\ell$ to an even power, and (2) we are lifting from an even power to an odd power of $\ell$, see Lemma \ref{nondis}.  We will show that lift-exceptional subgroups only arise when lifting from an even to an odd power of $\ell$ (Theorem \ref{upthm}).
Key to this argument is the fact that we are working with the entire isogeny class $\C(E)$; at each stage of the induction we replace $E$ by an isogenous curve in $\C(E)$ that has an $\ell^{n-1}$-isogeny itself.

Section \ref{sec_gpthy} contains the group-theoretic results necessary to prove the following theorem:
\begin{thm}\label{upthm}Let $E$ be an elliptic curve over a number field $K$.
Let $\ell$ be an odd prime and suppose that $\C(E)$ has a $K$-rational $\ell$-isogeny and locally almost everywhere has an $\ell^n$-isogeny.  Then:
\begin{itemize}
\item If $n=2$, $\C(E)$ has a $K$-rational $\ell^2$-isogeny,
\item If $n>2$, $\C(E)$ has a $K$-rational $\ell^n$-isogeny or, up to conjugacy, $G_{E'}(\ell^{2m+1})$ is contained in 
\[R(\ell^{2m+1}) = \left\{ \begin{pmatrix} r & s \\ \ell^{2m} (\epsilon s) & \epsilon t \end{pmatrix} \ : \ r\equiv t \pmod{\ell^{m+1}}, \epsilon = \pm 1 \right\}, \]
for some $0<m\leq (n-1)/2$ and some $E' \in \C(E)$.  In that case, $\C(E)$ has a $K$-rational $\ell^{2m}$-isogeny.
\end{itemize}
\end{thm}

The inductive proof of Theorem \ref{upthm} reduces to the following group-theoretic result:

\begin{thm}\label{upthm_group}
Let $\ell$ be an odd prime.  If $G \subset \GL_2(\zz/\ell^j \zz)$ is lift-exceptional at step $j$, then $j$ is odd and, up to conjugacy, $G$ is contained in $R(\ell^j) \subseteq \GL_2(\zz/\ell^j \zz)$.
\end{thm}

\begin{proof}[Theorem \ref{upthm_group} implies Theorem \ref{upthm}]
If \(\C(E)\) fails to have an \(\ell^n\)-isogeny, then it fails to have an \(\ell^j\)-isogeny for some minimal value of \(j > 1\).
Let \(E'\) be a curve in \(\C(E)\) that has an \(\ell^{j-1}\)-isogeny.  Then $G_E(\ell^{j-1})$ is contained in a Borel subgroup,
% (part \eqref{defin:lift_1} of Definition \ref{lift_defi}) 
\(\Delta(g)\) is a square modulo \(\ell^j\) for every element \(g \in G_E(\ell^{j})$,
and $G_E(\ell^{j})$ is not contained in any subgroup of the form \(A_{a,c,b}(\ell^{j+1})\).  In particular, $G_E(\ell^{j})$ is not contained in a Borel subgroup or a subgroup of the form \(A_{a, a+1, j-a-1}(\ell^j)\), and $G_E(\ell)$ is not contained in a split Cartan subgroup.  Hence $G_E(\ell^{j})$ is lift-exceptional and Theorem \ref{upthm_group} guarantees that $j$ is odd and, up to conjugacy, $G(\ell^j) \subset R(\ell^j)$.
\end{proof}

Let $G \subset \GL_2(\zz_\ell)$ be a closed subgroup such that $G(\ell^{n-1}) \subseteq B$ where $B$ is a fixed Borel subgroup, upper triangular in the chosen basis.
The proof of Theorem \ref{upthm_group} relies in a crucial way upon the following map
\[\phi  \colon G(\ell^{n-1}) \subseteq B \to \left( \zz/\ell^{n-1} \zz\right)^\times = \zz/\ell^{n-2}(\ell - 1) \zz,\]
defined as the ratio of the diagonal characters:
\[\phi \begin{pmatrix} \Ast_1 & \Ast \\ 0 & \Ast_2 \end{pmatrix} = \frac{\Ast_1}{\Ast_2}. \]
(Note that the map depends on the choice of $B$!)  

Under our inductive hypothesis that $G(\ell^{n-1})$ is contained in a Borel subgroup $B$, we consider the composite map
\[ \phi' \colon G(\ell^n) \to G(\ell^{n-1}) \xrightarrow{\phi} \left(\zz/\ell^{n-1}\zz\right)^\times, \]
where the first map is reduction mod $\ell^{n-1}$.
As $(\zz/\ell^{n-1}\zz)^\times$ is cyclic, the group $G(\ell^n)$ is generated by a preimage of a generator of the image of $\phi'$, and the kernel of $\phi'$.  To show that $G(\ell^n)$ is contained in a Borel, we use the local data that every element of $G(\ell^n)$ has square discriminant to show that each of these pieces is contained in \emph{the same} Borel $\tilde{B} \subset \GL_2(\zz/\ell^n\zz)$ lifting the Borel $B \subset \GL_2(\zz/\ell^{n-1}\zz)$ containing $G(\ell^{n-1})$.

This analysis is carried out in $\S$\ref{lemup}, with the result of proving a necessary but not sufficient condition on exceptional subgroups.  These ``potentially exceptional" subgroups are then analyzed more carefully in $\S$\ref{exceptional}, leading to a proof of Theorem \ref{upthm_group}.

\section{Group theoretic analysis of $\GL_2(\zz/\ell^n\zz)$ for odd $\ell$}\label{sec_gpthy}

In this section we will always assume that $\ell$ is an odd prime.

\subsection{Lifting Borel level structure}\label{lemup}

Let $G$ denote a closed subgroup of $\GL_2(\zz_\ell)$.  We will eventually take $G$ to be the image of the $\ell$-adic Galois representation of $E$.  Denote by $G(\ell^n)$ the reduction of $G$ mod $\ell^n$.  We say that $H \subset \GL_2(\zz/\ell^n\zz)$ \defi{lifts} $H' \subset \GL_2(\zz/\ell^{n-1}\zz)$ if $H \equiv H' \pmod{\ell^{n-1}}$.

\begin{lem}\label{hensel}
Let $\bar{\gamma} \in G(\ell)$ have $2$ distinct eigenvalues.  If $\gamma \in G(\ell^n)$ reduces mod $\ell$ to $\bar{\gamma}$ then $\gamma$ has $2$ distinct eigenvalues mod $\ell^n$ with corresponding eigenvectors.
\end{lem}

\begin{proof}
As the eigenvalues of $\bar{\gamma}$ are distinct ($\tr^2(\bar{\gamma}) - 4\det(\bar{\gamma}) \in \left(\zz/\ell\zz\right)^\times$), the derivatives of both the characteristic polynomial of $\bar{\gamma}$ evaluated on an eigenvalue and the projectivized eigenvector equation evaluated on a projectivized eigenvector are nonzero.  Hence Hensel's lemma guarantees the existence of distinct eigenvalues and corresponding eigenvectors of $\gamma$ mod $\ell^n$.
\end{proof}

\begin{lem}\label{dis}
Assume that $G(\ell^{n-1})$ is contained in a fixed Borel subgroup $B \subseteq \GL_2(\zz/\ell^{n-1} \zz)$.  Let $\gamma$ be an element of $G(\ell^n)$ whose reduction $\bar{\gamma} \in G(\ell)$ has distinct eigenvalues.
Then there exists some Borel subgroup $B' \subseteq \GL_2(\zz/\ell^n \zz)$ lifting $B$ which contains $\gamma$.
\end{lem}

\begin{proof}
Let $\gamma'$ denote the reduction of $\gamma$ mod $\ell^{n-1}$.
By Lemma \ref{hensel} $\gamma$ and $\gamma'$ each have two distinct eigenspaces.  Since $\gamma' \in B$ fixes only its two distinct eigenspaces, the group $B$ necessarily fixes one of the eigenspaces $L'$ of $\gamma'$, which is the reduction of an eigenspace $L$ of $\gamma$.  The Borel associated to $L$ lifts $B$ and contains $\gamma$.
\end{proof}

Recall that for $G(\ell^{n-1}) \subseteq B$, the map $\phi \colon G(\ell^{n-1}) \to \left(\zz/\ell^{n-1}\zz\right)^\times$ is defined as
\[\phi \begin{pmatrix} \Ast_1 & \Ast \\ &\Ast_2 \end{pmatrix} = \frac{\Ast_1}{\Ast_2}. \]
Define the subgroup $I(\ell^{n-1}) \subseteq G(\ell^{n-1})$ to be the kernel of the map $\phi$.  Let $K(\ell^n) \subseteq G(\ell^n)$ be the kernel of the composite map
\[\phi' \colon G(\ell^n) \to G(\ell^{n-1}) \to \left(\zz/\ell^{n-1}\zz\right)^\times, \]
which is alternatively the subgroup generated by the preimage of $I(\ell^{n-1})$ in $G(\ell^n)$:
\begin{equation}\label{Kdefin}
\begin{tikzcd}
K(\ell^n) \arrow[r, hook] & G(\ell^n) \arrow[d] \arrow[dr, dashed, "\phi' "] \\
I(\ell^{n-1}) \arrow[r, hook] & G(\ell^{n-1}) \arrow[r, "\phi"] & \left( \zz/ \ell^{n-1} \zz \right)^\times
\end{tikzcd}
\end{equation}

The group $G(\ell^n)$ is generated by $K(\ell^n)$ and a preimage $X\in G(\ell^n)$ of a generator of the image $\phi'(G(\ell^n))$.

\begin{lem}\label{liftX}
Assume that every element of $G(\ell^n)$ has square discriminant and that $G(\ell^{n-1})$ is contained in a Borel subgroup, but $G(\ell)$ is not contained in a split Cartan subgroup.  Assuming that the image of $\phi$ is nontrivial, there exists a preimage $X$ of a generator of the image of $\phi' \colon G(\ell^n) \to (\zz/\ell^{n-1} \zz)^\times$ that is contained in some Borel subgroup lifting the Borel mod $\ell^{n-1}$.
\end{lem}

\begin{proof}
We work in a basis in which $G(\ell^{n-1})$ is upper-triangular.  Let $X$ be some choice of preimage of a generator of the image of $\phi'$; by assumption $X$ has distinct diagonal entries mod $\ell^{n-1}$.  If $X$ has distinct diagonal entries mod $\ell$, the Lemma follows from Lemma \ref{dis}.  Otherwise $n \geq 3$ and $X$ has equal diagonal entries mod $\ell^j$ for some $0 < j \leq \lfloor \frac{n-1}{2} \rfloor$.  We will prove this by induction on $j$, showing that either the conclusions are satisfies or $X$ has equal diagonal entries modulo $\ell^{j+1}$.  Eventually, the argument will terminate for some $0 < j \leq \lfloor \frac{n-1}{2} \rfloor$.  Note that this implies that every element of $G(\ell^{n-1})$ has equal diagonal entries modulo $\ell^j$, and hence is of the form
\[\begin{pmatrix} a + \ell^j x & b \\ 0 & a + \ell^j w \end{pmatrix}. \]
Because we have
\begin{align*}
\begin{pmatrix} a + \ell^j x & b \\ 0 & a + \ell^j w \end{pmatrix} \vect{1}{k \ell^{n-j-1}} &= \vect{a + \ell^j x + bk \ell^{n-j-1}}{(a+ \ell^j w)k \ell^{n-j-1}} \\
&\equiv (a+ \ell^j(x + bk \ell^{n-2j-1}))\vect{1}{k \ell^{n-j-1}} \pmod{\ell^{n-1}}, 
\end{align*}
$G(\ell^{n-1})$ is contained in the intersection of all of the Borels fixing the lines spanned by $\vect{1}{k \ell^{n-j-1}}$ for any $k$.  The element $X \in G(\ell^n)$ is necessarily of the form
\[\begin{pmatrix} a + \ell^jx & b \\ \ell^{n-1} \zeta & a + \ell^j w \end{pmatrix}, \]
and $\ell \nmid b$ since $X$ is not contained in a Cartan mod $\ell$.  The discriminant of $X$
\[\Delta(X) = \ell^{2j} \left( (x-w)^2 + 4b\zeta \ell^{n-2j -1} \right), \]
is a square mod $\ell^{n}$ by assumption.  For $\vect{1}{k \ell^{n-j-1}}$ to be an eigenvector of $X$ for some $k$, we must have
\begin{align*}
k \ell^{n-j-1} ( a + \ell^j x + bk \ell^{n-j-1}) & \equiv \ell^{n-1} \zeta + (a+ \ell^j w) k \ell^{n-j-1} \pmod{{\ell^n}},\\
k\ell^j (x-w) + bk^2 \ell^{n-j-1} & \equiv \ell^j \zeta \pmod{\ell^{j+1}},   \\
k (x-w) + bk^2 \ell^{n-2j-1} & \equiv  \zeta \pmod{{\ell}} .
\end{align*}
\begin{itemize}
\item If $2j = n-1$, then this equation becomes
\[ b k^2 + (x-w)k -\zeta \equiv 0 \pmod{\ell}, \]
which has a solution in $k$ mod $\ell$ because $\ell \nmid b$ and the discriminant $(x-w)^2 + 4b \zeta$ is forced to be a square by the assumption that $\Delta(X)$ is a square.
\item If $2j < n-1 $, then this equation becomes
\[k(x-w) \equiv \zeta \pmod{\ell}. \]
This has a solution $k$ if $\ell \nmid (x-w)$ or $\ell \mid \zeta$.  Hence if $j$ is strictly less than $\lfloor \frac{n-1}{2} \rfloor$ then either there is a solution or $\ell \mid (x-w)$ and hence the diagonal characters are congruent modulo an even higher power of $\ell$.
\end{itemize}
It suffices to show that the induction on $j$ terminates.  If $n$ is odd, then it terminates at or before $j = \lfloor \frac{n-1}{2} \rfloor$, which was already dealt with above.  Otherwise, if $n = 2m$ then $2j < n-1$ and either the induction terminates, or when $j = m-1 = \lfloor \frac{2m-1}{2} \rfloor$ from above we have $\ell | (x-w)$.  Computing
\[ \Delta(X) = \ell^{2m-2} \left( (x-w)^2 + 4b\zeta \ell \right) \pmod{\ell^{2m}}. \]
This is a square modulo $\ell^{2m}$ if and only if
\[(x-w)^2 + 4b\zeta \ell \]
is a square modulo $\ell^2$.  But if $\ell \mid (x-w)$, then $4b \zeta \ell$ must be a square modulo $\ell^2$, and hence $\ell \mid b \zeta$, forcing $\ell \mid \zeta$, and $G(\ell^n)$ is contained in a Borel lifting the one corresponding to $k=0$. Thus the induction terminates at $j = \lfloor \frac{n-1}{2} \rfloor$ or before.
\end{proof}

Lemma \ref{liftX} guarantees the ability to lift the Borel structure mod $\ell^{n-1}$ to a Borel containing an appropriate choice of $X$, under suitable hypotheses.  We must now try to lift the Borel structure to one containing $K(\ell^n)$.  To do this, it is necessary to consider the cases of lifting from mod $\ell^{2m-1}$ to mod $\ell^{2m}$ and lifting from mod $\ell^{2m}$ to mod $\ell^{2m+1}$ separately. 

\begin{defin}
Let $\kk(\ell^{2m+1})$ for $m>0$ denote the subgroup of $\GL_2(\zz/\ell^{2m+1} \zz)$ given in an appropriate basis by
\[\kk \colonequals   \kk(\ell^{2m+1}) = \left \{ \begin{pmatrix} r & s \\ \ell^{2m}s & t \end{pmatrix} \ : \ r \equiv t \pmod{\ell^{2m}} \right\}. \]
\end{defin}

%\begin{lem}\label{sq_disc}
%Suppose that
%\[H \subseteq \left\{ \begin{pmatrix} x & y \\ \ell^nz & x  + \ell^aw \end{pmatrix} \right\} \subset \GL_2(\zz/\ell^n\zz),\]
%for some \(a > \lfloor (n-1)/2 \rfloor\),
%and every element of \(H\) has square discriminant.
%If \(H\) contains an element for which \(y \not\equiv 0 \pmod{\ell}\), as well as an element for which 
%\end{lem}

\begin{lem}\label{nondis}
Assume that \(G(\ell^n)\) satisfies Definition \ref{lift_defi}\eqref{defin:lift_1}--\eqref{defin:lift_3} and \eqref{defin:lift_5}.  Suppose that \(G(\ell^{n-1})\) is contained in a fixed Borel \(B\) and 
%\begin{itemize}
%\item \(G(\ell^{n-1})\) is contained in a Borel \(B\);
%\item every element of \(G(\ell^n)\) has square discriminant;
%\item \(G(\ell)\) is not contained in split Cartan, and \(G(\ell^n)\) is not conjugate to a subgroup of \(A_{a,a+1, n-a-1}(\ell^n)\) for any \(0 \leq a \leq \lfloor (n-1)/2 \rfloor\) 
%\end{itemize}
let $I(\ell^{n-1})$ and $K(\ell^n)$ be as in \eqref{Kdefin}. 
\begin{enumerate}[(i)]
\item If $n = 2m$ is even, then $K(\ell^{2m})$ is contained in the intersection of all Borel subgroups of $\GL_2(\zz/\ell^{2m} \zz)$ lifting $B$ (and hence \(G(\ell^{2m})\) is contained in some Borel subgroup lifting \(B\).)
\item If $n=2m+1$ is odd and \(K(\ell^{2m+1}) \mod \ell\) is contained in a split Cartan subgroup, then either
\begin{itemize} 
\item $G(\ell^{2m+1})$ is contained in the intersection of all Borel subgroups of $\GL_2(\zz/\ell^{2m} \zz)$ lifting $B$, or
\item \(G(\ell^{2m+1}) \subseteq \kk(\ell^{2m+1}) \subset R(\ell^{2m+1})\), in an appropriate basis. 
\end{itemize}
\item If $n=2m+1$ is odd and \(K(\ell^{2m+1}) \mod \ell\) is not contained in a split Cartan subgroup, then either
\begin{itemize} 
\item $K(\ell^{2m+1})$ is contained in the intersection of all Borel subgroups of $\GL_2(\zz/\ell^{2m} \zz)$ lifting $B$ (and hence \(G(\ell^{2m+1})\) is contained in some Borel subgroup lifting \(B\)), or
\item $G(\ell)$ is contained in a radical subgroup and 
\[K(\ell^{2m+1}) \subseteq \kk(\ell^{2m+1}), \]
in an appropriate basis. 
\end{itemize}
\end{enumerate}
\end{lem}

\begin{proof}
We will work in the basis for $\left(\zz/\ell^{n}\zz\right)^2$ such that some Borel $B' \subset \GL_2(\zz/\ell^n\zz)$ lifting $B$ is upper-triangular.
Any element of $K(\ell^{n})$ is of the form
\[ \begin{pmatrix} k_1 & k_2 \\ \ell^{n-1}k_3 & k_1 + \ell^{n-1}w \end{pmatrix}, \]
which has discriminant 
\begin{equation}\label{discriminant_eq} 4\ell^{n-1}k_2 k_3 \pmod{\ell^{n}}. \end{equation}

\noindent
\textbf{\boldmath Case 1.} We first deal with the case that there exists some element of \(K(\ell^n)\) that has \(k_2 \not\equiv 0 \pmod{\ell}\), and hence \(K(\ell^n)\) is generated by elements with this property.  

If $n=2m$ is even, then $\ell^{2m-1}$ is not a square mod $\ell^{2m}$.  Hence, considering the discriminant \eqref{discriminant_eq} of any generator, as $\ell \nmid k_2$, we must have that $k_3 \equiv 0 \pmod \ell$.  As such, all generators of $K(\ell^{n-1})$ are necessarily upper-triangular in this basis and as such contained in $B'$.  Thus $K(\ell^n)$ is contained in the intersection of all such Borel subgroups lifting $B$.  Recall that \(G(\ell^n)\) is generated by \(K(\ell^n)\) and one preimage \(X\) of a generator of the image of \(\phi'\).  By Lemma \ref{liftX}, we can assume $X$ is contained in some Borel $B' \subseteq \GL_2(\zz/\ell^{n}\zz)$ lifting $B$, or $G(\ell^{n}) \subset K(\ell^{n})$.  In either case, \(G(\ell^n)\) is contained in a Borel subgroup.

If $n=2m+1$ is odd, then $\ell^{2m}$ is a square mod $\ell^{2m+1}$ and we must have $k_2 k_3$ a square mod $\ell$ for all elements of $K(\ell^n)$.  We will show that if $G(\ell)$ is not contained in a radical subgroup, then the Borel structure mod $\ell^{n-1}$ must lift to mod $\ell^n$ in order for $K(\ell^n)$ to be a normal subgroup of $G(\ell^n)$.  

For the discriminant \eqref{discriminant_eq} to be a square mod $\ell^n$, we are equivalently concerned with whether $\beta\zeta$ 
is a quadratic residue mod $\ell$, where
\[\beta = \frac{k_2}{k_1} \mod \ell, \qquad \zeta = \frac{k_3}{k_1} \mod \ell. \]
The operation of matrix multiplication on elements of $K(\ell^{2m+1})$:
%\[ \begin{pmatrix} a + \ell^{n-1}x & b + \ell^{n-1}y \\ \ell^{n-1}z & a + \ell^{n-1}w \end{pmatrix} \cdot \begin{pmatrix} a' + \ell^{n-1}x' & b' + \ell^{n-1}y' \\ \ell^{n-1}z' & a' + \ell^{n-1}w' \end{pmatrix}\]
acts as \((\beta, \zeta) + (\beta', \zeta') = (\beta + \beta', \zeta + \zeta')\).
%\[ \left( \frac{b}{a}, \frac{z}{a} \right) \cdot \left( \frac{b'}{a'}, \frac{z'}{a'} \right) = \left( \frac{b}{a} + \frac{b'}{a'}, \frac{z}{a} + \frac{z'}{a'} \right).\]
As such the set of possible $(\beta, \zeta)$ is a nontrivial additive subgroup of $\ff_\ell^2$, in which $\beta\zeta$ is necessarily a quadratic residue.  Hence it is necessarily a line $(\delta x, \gamma x) \ \forall x \in \ff_\ell$ in which 
\[\delta \gamma = \frac{\beta \zeta}{x^2} \]
is a quadratic residue.

If $G(\ell)$ is not contained in a radical subgroup then it contains some element $G = \begin{pmatrix} g & \ast \\ 0 & f \end{pmatrix}$,
where $f^2 \neq g^2$.
The action of conjugation by (a lift to $G(\ell^{2m+1})$ of) $G$ defines an operator on the group of $(\beta, \zeta)$ whose action is
\[(\beta, \zeta) \mapsto \left( \frac{f}{g} \beta, \frac{g}{f} \zeta \right). \]
In order for this to stabilize the line (necessary since $K(\ell^n)$ is a normal subgroup), either \(\beta\) is identically zero (impossible since we assumed that \(K(\ell^n)\) has and element with \(k_2 \not\equiv 0 \pmod{\ell}\)), or $\zeta$ is identically zero (which implies that \(K(\ell^n) \subset B'\) as desired) since $f^2 \neq g^2$ by assumption.  As above, if \(K(\ell^n) \subset B'\), then \(G(\ell^n)\) is contained in some Borel subgroup.

Otherwise we have that $G(\ell)$ is contained in a radical subgroup.  Further, for all elements of $K(\ell^n)$, $k_2k_3$ is a square mod $\ell$ and $k_2/k_3$ is a fixed element of $\ff_\ell^\times$ (determining this line in $\ff_\ell^2$).  Up to conjugation (e.g. in an appropriate basis), we may assume that this ratio is $1$.  Hence $K(\ell^{2m+1})$ is contained in $\kk(\ell^{2m+1})$.

\noindent
\textbf{\boldmath Case 2.} We now turn to the case that every element of \(K(\ell^n)\) has \(k_2 \equiv 0 \pmod{\ell}\)  (in particular, \(K(\ell^n) \mod \ell\) is contained in a split Cartan subgroup.)  Notice that \(K(\ell^n)\) is contained in every subgroup \(A_{a, a+1, n-a-1}(\ell^n)\) for \(a \leq n-2\) by construction.  We must therefore have that the image of \(\phi\) is nontrivial mod \(\ell^{n-1}\), or it contradicts our assumption that Definition \ref{lift_defi}\eqref{defin:lift_5} holds.  Notice also that
\(K(\ell^n) \pmod{\ell}\) is scalar, and hence contained in \emph{every} split Cartan subgroup of \(\GL_2(\zz/\ell \zz)\).  

If the image of \(\phi'\) is nontrivial modulo \(\ell\), then \(X\) mod \(\ell\) is uppertriangular with distinct diagonal entries.  Therefore, \(X\) mod \(\ell\) is contained in \emph{some} split Cartan subgroup of \(\GL_2(\zz/\ell \zz)\); together with \(K(\ell^n) \pmod{\ell}\), we see that all of \(G(\ell)\) is contained in a split Cartan subgroup, which contradicts our assumption that Definition \ref{lift_defi}\eqref{defin:lift_3} holds.

We may therefore assume that the image of \(\phi\) is trivial modulo some power \(\ell^a\) and nontrivial modulo \(\ell^{a+1}\) for some \(1 \leq a\leq n-2\).  If \(a > \lfloor (n-1)/2 \rfloor\), then we may apply Case 1 to all of \(G(\ell^n)\) instead of the normal subgroup \(K(\ell^n)\).  Note that \(G(\ell)\) always contains elements with \(k_2 \not\equiv 0 \pmod{\ell}\), otherwise it would contradicts our assumption that Definition \ref{lift_defi}\eqref{defin:lift_3} holds.

It remains to treat the case \(a \leq \lfloor (n-1)/2 \rfloor\).  The element \(X\) is of the form
\[X = \begin{pmatrix} x & y \\ \ell^n z & x + \ell^a w \end{pmatrix}.\]
If \(y \equiv 0 \pmod{\ell}\), then \(X\), and hence all of \(G(\ell)\), is contained in a Cartan subgroup.  We therefore assume that \(\ell \nmid y\).  Let \(\nu = -w/y\) (which is nonzero by our assumption that the diagonal entries of \(X\) are distinct modulo \(\ell^{a+1}\)).  Conjugating \(A_{a, a+1, n-a-1}(\ell^n)\) by the element \(\sm{1}{}{}{\nu}\) yields the group
\[  A^\nu_{a, a+1, n-a-1}(\ell^n) = \left\{ \begin{pmatrix} \alpha & \beta \\ \gamma & \delta \end{pmatrix} : \gamma \equiv 0 \pmod{\ell^{n-1}}, \delta \equiv \alpha - \ell^a \nu\beta \pmod{\ell^{a+1}} \right\},\]
to which \(X\) evidently belongs.  Since we assumed that \(K(\ell^n) \mod \ell\) is contained in a split Cartan subgroup, we also have that \(K(\ell^n) \subset A^\nu_{a, a+1, n-a-1}(\ell^n)\), and so the same is true for \(G(\ell^n)\).
\end{proof}

\begin{defin}
We say that $H \subseteq \GL_2(\zz/\ell^{2m+1}\zz)$, $m>0$, is \defi{potentially lift-exceptional (at step $2m+1$)} if 
\begin{itemize}
\item $H(\ell^{2m})$ is contained in a Borel subgroup of $\GL_2(\zz/\ell^{2m}\zz)$, and
\item For all $h \in H$, the discriminant $\Delta(h) \in \zz/\ell^{2m+1}\zz$ is a square, and 
\item $H(\ell)$ is contained in a radical subgroup but not a split Cartan subgroup, and
\item $\ker\big(\phi' \colon H \to (\zz/\ell^{2m} \zz)^\times\big)$ is contained in $\kk(\ell^{2m+1})$. 
\end{itemize}
\end{defin}

By Lemma \ref{nondis}, $H$ is lift-exceptional at step $n$ only if $n = 2m+1$ is odd and it is potentially lift-exceptional at step $2m+1$.  Furthermore, Lemma \ref{nondis} immediately implies Theorem \ref{upthm_group} for any lift-exceptional group with \(K(\ell^n)\) contained in a split Cartan modulo \(\ell\).  We therefore restrict to groups for which this is not the case.  We say that a subgroup \(K \subset \kk(\ell^{2m+1})\) is \defi{non-Cartan modulo \(\ell\)} if it contains an element with \(s \not\equiv 0 \pmod{\ell}\).  The map \(\psi \colon \kk(\ell^{2m+1}) \to \ff_\ell\) extracting the value of \(s/r\) modulo \(\ell\) is a group homomorphism.  In particular, any non-Cartan subgroup of \(\kk(\ell^{2m+1})\) surjects onto \(\ff_\ell\) under \(\psi\), and we use this for the remainder of the paper.

%\begin{lem}
%Let $H \subset \GL_2(\zz/\ell^{n}\zz)$.  If $H(\ell)$ is not contained in a split Cartan subgroup, then $H$ is lift-exceptional at step $n$ only if $n = 2m+1$ is odd and it is potentially lift-exceptional at step $2m+1$.
%\end{lem}
%\begin{proof}
%Fix a Borel $B \subset \GL_2(\zz/\ell^{n-1}\zz)$ such that $H(\ell^{n-1}) \subset B$.
%Recall that $K(\ell^{n})$ is the kernel of the map $\phi' \colon G(\ell^{n}) \to (\zz/\ell^{n-1}\zz)^\times$, and $X$ is a preimage of a generator of the image.  Together they generate $G(\ell^{n})$.
%
%By Lemma \ref{liftX}, we can assume $X$ is contained in some Borel $B' \subseteq \GL_2(\zz/\ell^{n}\zz)$ lifting $B$, or $G(\ell^{n}) \subset K(\ell^{n})$.  But further by Lemma \ref{nondis} either $K(\ell^n)$ is contained in the intersection of all Borel subgroups of $\GL_2(\zz/\ell^n\zz)$ lifting $B$, or $n=2m+1$ is odd.  In the first case, the generators of $G(\ell^n)$ are contained in the Borel subgroup $B'$, implying the same of $G(\ell^n)$. 
%In the second case, by Lemma \ref{nondis}
%$G(\ell)$ is contained in a radical subgroup and $K(\ell^{2m+1}) \subseteq \kk(\ell^{2m+1})$.
%\end{proof}

In the next section, we consider exactly when groups with $G(\ell)$ radical and $K(\ell^{2m+1}) \subseteq \kk(\ell^{2m+1})$ non-Cartan mod \(\ell\) give rise to lift-exceptional subgroups.  This will lead to the classification cited in the introduction and to the proof of Theorem \ref{upthm_group}.

\subsection{Lift-exceptional subgroups}\label{exceptional}

The goal of this section to to determine which potentially lift-exceptional subgroups are actually lift-exceptional.

For ease of notation we will drop the degree from $\kk(\ell^{2m+1})$ and refer to it simply as $\kk$; the fact that it is a subgroup of $\GL_2(\zz/\ell^{2m+1}\zz)$ will be implied.

\begin{lem}\label{kevec}
A vector $v \in (\zz/\ell^{2m+1} \zz)^2$ is a simultaneous eigenvector for all elements of $\kk$ precisely when $v$ is of the form
\[ v = \lambda \begin{pmatrix} 1 \\ k \ell \end{pmatrix},\]
for any $\lambda \in \left(\zz/\ell^{2m+1}\zz \right)^*$ and $k \equiv \pm \ell^{m-1} \pmod{\ell^{2m-1}}$.
\end{lem}

\begin{proof}
The vector $v$ can be normalized, multiplication by an invertible scalar, to be of the form $\begin{pmatrix} 1 \\ k \ell \end{pmatrix}$ since $\kk \mod \ell$ contains the matrix $\begin{pmatrix} 1 & 1 \\ 0 & 1 \end{pmatrix}$.  The condition that such a vector is an eigenvector for any matrix $\sm{r}{s}{\ell^{2m}s}{t}$ in $\kk$ reduces to
\[ k^2s \equiv \ell^{2m-2} s \pmod{\ell^{2m-1}}. \]
As $\kk$ contains matrices with $\ell \nmid s$, we conclude that necessarily $k\equiv \pm \ell^{m-1} \pmod{\ell^{2m-1}}$.
\end{proof}

This suggests a method of understanding exceptional subgroups that arise from the inability to ``obviously" lift the Borel structure on the matrices with equal diagonal entries mod $\ell^{2m}$.  If the matrix $X$, which is a preimage under $\phi'$ of a generator of the image of $\phi'$, also has one of these vectors as an eigenvector, then the group generated by $X$ and $\kk$ -- what we will denote $X \cdot \kk$ \footnote{In our context $\kk$ is a normal subgroup of this composite, so every element of the group can be represented as a power of $X$ times an element of $\kk$.  We will always work under this assumption on $X$.} -- is contained in the corresponding Borel subgroup, and hence is not exceptional.  Otherwise, if every element of $X \cdot \kk$ has square discriminant, but $X$ shares no eigenvectors with all elements of $\kk$, then $X \cdot \kk$ is lift-exceptional.  Note that a subgroup of $X \cdot \kk$ may fail to be lift-exceptional even when $X \cdot \kk$ is.

\begin{lem}\label{borelcomp}
Let $X \in \GL_2(\zz/\ell^{2m+1}\zz)$ be contained in a Borel subgroup mod $\ell^{2m}$.
Then $X \cdot \kk$ is contained in a Borel subgroup mod $\ell^{2m+1}$ if and only if $X$ has equal diagonal characters mod $\ell^{m}$ and writing
\[X = \begin{pmatrix} a + \ell^{m} x & b  \\ \ell^{2m}z & a + \ell^{m} w \end{pmatrix}, \]
we have
\[\pm (x-w) \equiv (z-b) \pmod{\ell}.\]
\end{lem}

\begin{proof}
In a suitable basis, it suffices by Lemma \ref{kevec} to consider when $X = \begin{pmatrix} x_1 & b \\ \ell^{2m}z & x_2 \end{pmatrix}$ stabilizes the line spanned by $\vect{1}{k \ell}$ for $k \equiv \pm \ell^{m-1} \pmod{\ell^{2m-1}}$.  The eigenvector equation reduces to
\[\ell^{2m-1}z + k x_2 \equiv k x_1 + k^2 \ell b \pmod{\ell^{2m}}. \]
A consideration of valuations of the monomials gives that $\ell^m \mid (x_1-x_2)$.  Now letting $x_1 = a + \ell^m x$ and $x_2 = a+ \ell^m w$, the above simplifies to $(x-w) \equiv \pm (z-b) \pmod{\ell}$ as desired.
\end{proof}

\begin{lem}\label{inductcong}
Let $Y \in \GL_2(\zz/\ell^{2m+1} \zz)$ be contained in a Borel subgroup mod $\ell^{2m}$ with equal diagonal entires modulo $\ell^j$ for $0 < j < m$.  Let \(K \subset \kk\) be any subgroup that is non-Cartan mod \(\ell\).  If every element of $Y \cdot K$ has square discriminant then $Y$ has equal diagonal entries modulo $\ell^{j+1}$.
\end{lem}

\begin{proof}
By hypothesis, every element of $Y \cdot K$ is a power of $Y$ times an element of $K$.   
Write
\[Y = \begin{pmatrix} a + \ell^j x & b \\ \ell^{2m}z & a + \ell^j w \end{pmatrix}, \text{ and}\qquad  R = \begin{pmatrix} r & s \\ \ell^{2m}s & t \end{pmatrix}\]
 for an element of $K$ (recall that $r \equiv t \pmod{\ell^{2m}}$).  It is easy to prove by induction that
\[ Y^k = \begin{pmatrix} (a+\ell^jx)^k + {k \choose 2} \ell^{2m}a^{k-2}zb & k a^{k-1}b + \ell(\Ast) \\ k \ell^{2m}za^{k-1} & (a+\ell^j w)^k + {k \choose 2} \ell^{2m}a^{k-2} zb \end{pmatrix},\]
where $\Ast$ is an element of $\zz/\ell^{2m}\zz$, whose precise value is not important for the discriminant computation we are undertaking.
Now consider the product
\[Y^{\ell^{m-j}} \cdot R = \begin{pmatrix} (a+\ell^jx)^{\ell^{m-j}} &\ell(\Ast') \\ 0 & (a+\ell^jw)^{\ell^{m-j}} \end{pmatrix} \begin{pmatrix} r & s \\ \ell^{2m}s & t \end{pmatrix} = \begin{pmatrix} r(a+\ell^jx)^{\ell^{m-j}} & a^{\ell^{m-j}}s + \ell(\Ast'') \\ \ell^{2m}a^{\ell^{m-j}}s & t(a+\ell^jw)^{\ell^{m-j}} \end{pmatrix}, \]
where again $\Ast'$ and $\Ast''$ are elements of $\zz/\ell^{2m}\zz$.
This element has discriminant
\begin{align*}
\Delta(Y^{\ell^{m-j}} \cdot R) &= \left( r(a+\ell^j x)^{\ell^{m-j}} - t(a+\ell^jw)^{\ell^{m-j}} \right)^2 + 4\ell^{2m}a^{2\ell^{m-j}}s^2 \\
&= \left( \sum_{k=0}^{\ell^{m-j}} {\ell^{m-j} \choose k} \ell^{jk} a^{\ell^{m-j}-k} (rx^k-tw^k) \right)^2 +4\ell^{2m}a^{2\ell^{m-j}}s^2 \\
&\equiv \left( r\ell^ma^{\ell^{m-j}-1}(x-w) + \sum_{k=2}^{\ell^{m-j}} {\ell^{m-j} \choose k} \ell^{jk} ra^{\ell^{m-j}-k} (x^k-w^k) \right)^2 +4\ell^{2m}a^{2\ell^{m-j}}s^2\\
&\qquad \qquad\qquad \qquad\qquad \qquad \qquad\qquad \qquad \qquad  \qquad\qquad \qquad\qquad \pmod{\ell^{2m+1}},\\
\end{align*}
since every term in sum includes some power of $\ell$, and $r\equiv t \pmod{\ell^{2m}}$.  Notice that the minimal $\ell$-adic valuation of any term of the above expression is $2m$.
By Hensel's lemma (and our assumption that $\ell \neq 2$) the discriminant $\Delta(Y^{\ell^{m-j}} \cdot R)$ is a square modulo $\ell^{2m+1}$ if and only if it is zero or an even power of $\ell$ times a (nonzero) quadratic residue mod $\ell$.  Hence the above discriminant is a square if and only if
\[r^2a^{2\ell^{m-j}-2}(x-w)^2 + 4a^{2\ell^{m-j}}s^2 = a^{2\ell^{m-j}-2}(r^2(x-w)^2 + 4a^2s^2), \]
is a square mod $\ell$.  In order for \emph{every} element of $Y \cdot K$ to have square discriminant, the quadratic form
\[Q(r,s) = (x-w)^2/r^2 + 4a^2s^2 \]
must be zero or a quadratic residue modulo $\ell$ for all values of $s/r \in \ff_\ell$.  (We are using here that for non-Cartan \(K\), all possible values of \(s/r\) are achieved.)  Hence the discriminant of this form, $-16(x-w)^2a^2$, must be zero mod $\ell$ (any form with nonzero discriminant represents a quadratic non-residue).  As $\ell \nmid a$, we have $\ell \mid (x-w)$.  So in fact the diagonal characters are equal modulo $\ell^{j+1}$.  \end{proof}

\begin{lem}\label{lastcong}
Let $Y \in \GL_2(\zz/\ell^{2m+1}\zz)$ be of the form $\begin{pmatrix} a + \ell^mx & b \\ \ell^{2m}z & a + \ell^m w \end{pmatrix}$.  Let \(K \subset \kk\) be any subgroup that is non-Cartan mod \(\ell\).  If every element of $Y \cdot K$ has square discriminant then 
\[ (x-w) \equiv \pm(z-b) \pmod{\ell}. \]
\end{lem}

\begin{proof}
As above, let $R = \begin{pmatrix} r  & s \\ \ell^{2m}s & t \end{pmatrix}$ be an element of $K$.  Then we have
\begin{align*}
\Delta(Y) &=  \ell^{2m}\left((x-w)^2 + 4 bz\right), \\
\Delta(Y \cdot R) &= \ell^{2m}r^2(x-w)^2 + 4 \ell^{2m}(zr + as)(as+bt) \\
&= r^2 \Delta(Y) + 4\ell^{2m}(a^2s^2 + asr(z+b)).
\end{align*}
If $\Delta(Y)$ is a square, then it is $\ell^{2m} \delta^2$, for some $\delta \in \zz/\ell\zz$.  In order for $\Delta(Y \cdot R) = \ell^{2m}(\delta^2r^2 + 4a(z+b) rs + 4a^2 s^2)$ to be square, we must have that the quadratic form
\[Q(r,s) = \delta^2r^2 + 4a(z+b) rs + 4a^2 s^2 \]
does not represent any quadratic nonresidue of $\zz/\ell\zz$.  Hence (using that all possible values of \(s/r\) are achieved) the discriminant of $Q$ must be $0$ mod $\ell$:
\[(z+b)^2 \equiv (x-w)^2 + 4zb \pmod{\ell},\]
which implies
\[ (x-w) \equiv \pm (z-b) \pmod{\ell}, \]
as desired.
\end{proof}

\begin{prop}\label{excepX}
Let $X$ be radical mod $\ell$ and upper-triangular mod $\ell^{2m}$, but not diagonal mod $\ell$.  Let \(K \subset \kk\) be any subgroup that is non-Cartan mod \(\ell\).  Then $X \cdot K$ is lift-exceptional (at step $2m+1$) only if the diagonal characters of $X$ are opposite modulo $\ell^{m+1}$.  Furthermore, if the diagonal characters of $X$ are opposite modulo $\ell^{m+1}$, then $X \cdot \kk$ is lift-exceptional (at step $2m+1$)
\end{prop}

\begin{proof}
In order to be lift-exceptional (at step $2m+1$), $X \cdot K$ must not be contained in a Borel subgroup and every element of $X \cdot K$ must have square discriminant.  We consider two cases separately:

\textbf{$X$ has equal diagonal characters mod $\ell$:} in this case, write
\[ X = \begin{pmatrix} a + \ell^j x & b \\ \ell^{2m}z & a + \ell^j w \end{pmatrix}, \text{ and}\qquad  R = \begin{pmatrix} r & s \\ \ell^{2m}s & t \end{pmatrix},\]
for some $j \geq 1$.  
Lemma \ref{borelcomp} shows that $X \cdot \kk$ is contained in a Borel subgroup if and only if $j=m$ and 
\begin{equation}\label{fundcong} (x-w) \equiv \pm (z+b) \pmod{\ell}. \end{equation}
But Lemmas \ref{inductcong} and \ref{lastcong} with $Y = X$ show that the congruence \eqref{fundcong} must hold if every element of $X \cdot K$ has square discriminant.  So $X \cdot K$ is never lift-exceptional.

\textbf{$X$ has opposite diagonal characters mod $\ell$:} Notice that the entire group \(\kk(\ell^{2m+1})\) is not conjugate to any subgroup of \(A_{a, a+1, n-a-1}(\ell^n)\) for any \(0 \leq a < m\).  The group \(\kk(\ell^{2m+1})\) is also not contained in a split Cartan subgroup modulo \(\ell\).  For this reason, to check if the group \(X \cdot \kk\) is lift-exceptional, it suffices to check that it is not contained in a Borel subgroup (i.e., Definition \ref{lift_defi} \eqref{defin:lift_3}\&\eqref{defin:lift_5} are impossible). Lemma \ref{borelcomp} shows that $X \cdot \kk$ is never contained in a Borel subgroup.  

It suffices to show that every element of $X \cdot K$ has square discriminant precisely when $X$ has opposite diagonal characters modulo $\ell^{m+1}$.  
If $X$ has opposite diagonal characters modulo $\ell^j$, then $X^2$ has equal diagonal characters modulo $\ell^j$.  Hence Lemma \ref{inductcong} applied to $Y = X^2$ shows that this assumption forces the diagonal characters to be opposite modulo $\ell^m$ (since the sum and difference of the diagonal characters of $X$ can't both be divisible by $\ell$ because $X$ is contained in a Borel and invertible mod $\ell$).  Similarly, Lemma \ref{lastcong} applied to $Y = X^2$ shows that the diagonal characters are opposite modulo $\ell^{m+1}$, since the ``$z$" and ``$b$" of $Y = X^2$ are both $0$ mod $\ell$.  So this congruence condition is necessary for $X \cdot K$ to be lift-exceptional.
It is an easy calculation that all elements of $X \cdot \kk$ have square discriminant if this congruence condition is met. 
\end{proof}

\subsection{Proof of Theorem \ref{upthm_group}}

To prove Theorem \ref{upthm_group}, it suffices to prove that for any choice of $X$ subject to the constraints of Proposition \ref{excepX}, the exceptional subgroup generated is contained in the group
\[R(\ell^{2m+1}) = \left\{\begin{pmatrix} r & s \\  \ell^{2m}(\epsilon s) & \epsilon t \end{pmatrix} \ : \ {\epsilon = \pm 1 \atop r \equiv t \pmod{\ell^{m+1}} } \right\}, \]
in an appropriate basis.

\begin{prop}
For any $X$ of the form
\[X = \begin{pmatrix} a + \ell^{m+1}x & b \\ \ell^{2m}z & -a + \ell^{m+1}w \end{pmatrix}, \]
$X \cdot \kk$ is contained in $R(\ell^{2m+1})$  up to conjugation.
\end{prop}

\begin{proof}
If $X$ has the above form with $z \equiv -b \pmod{\ell}$, then all of its powers will too.  Hence it suffices to check that something of this form times an arbitrary element of $\kk$ is contained in $R(\ell^{2m+1})$:
\[\begin{pmatrix} a + \ell^{m+1}x & b \\ -\ell^{2m}b & -a + \ell^{m+1}w \end{pmatrix} \cdot \begin{pmatrix} r & s \\ \ell^{2m}s & t \end{pmatrix} = \begin{pmatrix} r(a+\ell^{m+1}) + \ell^{2m}bs & s(a + \ell^{m+1}x) + bt \\ \ell^{2m}(-sa-br) & t(-a + \ell^{m+1})-\ell^{2m}sb \end{pmatrix}, \]
which is clearly true.

Now it remains only to show that we may conjugate $X$ and $\kk$ by an appropriate matrix $M$ so as to bring $\kk$ into itself and bring $X' = M^{-1}XM$ into another matrix of the same form except with off-diagonal entries of $X'$ so that $z' \equiv -b' \pmod{\ell}$.  This is achieved by the a matrix of the form $M = \begin{pmatrix} 1 & \mu \\ 0 & 1 \end{pmatrix}$
as we have that:
\[M^{-1} \cdot X \cdot M = \begin{pmatrix} a + \ell^{m+1}x - \ell^{2m}z\mu & 2\mu  a + b + \ell^{m+1}\mu(x-w) - \ell^{2m}\mu^2z \\ \ell^{2m}z & -a + \ell^{m+1} w + \ell^{2m}z\mu \end{pmatrix}. \]
As $a$ is a unit, we may choose $\mu \equiv \frac{z-b}{2a} \pmod{\ell}$.

Finally, on some element $R$ of $\kk$ this acts as
\[ M^{-1} \cdot R \cdot M = \begin{pmatrix} r - \ell^{2m}s\mu & s + \mu(r-t) - \ell^{2m}s\mu^2 \\ \ell^{2m}s & t + \ell^{2m}s \mu \end{pmatrix}, \]
which is again an element of $\kk$ as $\ell^{2m} \mid (r-t)$.
\end{proof}

\begin{rem}
$R(\ell^{2m+1})$ is itself exceptional and corresponds to $X \cdot \kk(\ell^{2m+1})$ for
\[X = \begin{pmatrix} 1 & 0 \\ 0 & -1 + \ell^{m+1} \end{pmatrix}. \]
\end{rem}

We now prove the last part of Proposition \ref{ellnthm_group}. 

\begin{cor}\label{quadextn}
Let $E$ be an elliptic curve over $K$.  If $\C(E)$ has an $\ell^{2m}$-isogeny but fails to have an $\ell^{2m+1}$-isogeny, then there exists some quadratic extension $F/K$ such that the isogeny class $\C(E_F)$ of the base change of $E$ to $F$ has an $\ell^{2m+1}$-isogeny.
\end{cor}
\begin{proof}
The index $2$ subgroup of $R(\ell^{2m+1})$ where $\epsilon = +1$ is contained in the Borel subgroup corresponding to the vector $\vect{1}{\ell^m}$.
\end{proof}

\begin{cor}\label{counter}
For every $\ell$ and every $m$, there exists a number field $K$ and an elliptic curve $E$ over $K$ such that $\C(E)$ has an $\ell$-isogeny and locally almost everywhere has an $\ell^{2m+1}$-isogeny, but does not have an $\ell^{2m+1}$-isogeny.
\end{cor}
\begin{proof}
It suffices to exhibit a subgroup $H$ of $\GL_2(\zz/\ell^{2m+1} \zz)$ that is lift-exceptional at step $2m+1$ for every $\ell$.  Let $E'/K'$ be an elliptic curve over a number field with surjective $\ell$-adic Galois representation.  Let $K = K'(E'[\ell^{2m+1}])^H$.  By Galois theory, the base change $E = E'_{K}$ has $G_E(\ell^{2m+1}) = H$.
As noted in the above Remark, we can simply take $H = R(\ell^{2m+1})$.
\end{proof}

\begin{rem}
If instead of asking for a local-global principle for the isogeny class of $E$ we asked for a strict local-to-global principle, our potentially lift-exceptional restrictions would still give a necessary condition on exceptional subgroups if $G(\ell)$ is not contained in a split Cartan subgroup.  The following example shows that the set of exceptional subgroups is still non-empty even when you add the extra structure of $E$ itself having an $\ell^n$-isogeny locally almost everywhere:

Let $H \subseteq \GL_2(\zz/\ell^3 \zz)$ be the subgroup generated by
\[M = \begin{pmatrix} 1 & 1 \\ \ell^2 & 1 \end{pmatrix}, \qquad X = \begin{pmatrix} 1 & 0\\0 & -1 \end{pmatrix}. \]
$H$ is not contained in a Borel subgroup mod $\ell^3$ as the only two Borel subgroups containing $X$ are the upper and lower-triangular matrices.  $M$ itself is contained in a Borel, as
\[ \begin{pmatrix} 1 & 1  \\ \ell^2 & 1 \end{pmatrix} \begin{pmatrix} 1 \\ \ell \end{pmatrix} = (1+ \ell)\begin{pmatrix} 1  \\ \ell \end{pmatrix}. \]
Hence $H$ cannot be the image of the mod $\ell^3$ Galois representation of an elliptic curve $E$ satisfying the strict conclusions of the local-to-global principle ($\C(E)$ also fails the conclusions of the local-global principle).  All that remains is to show that it satisfies the hypotheses -- that is that every element is contained in a Borel subgroup.  

To show this, we may replace $H$ with the quotient by all scalar matrices $\pp H$.
In this case we have the relations $M^{\ell^3} = 1, X^2 = 1, XMX = M^{-1}$.  This implies that we have an exact sequence
\[0 \to \langle M \rangle \to \pp H \to \zz/2\zz \to 0, \]
where the last map counts the number of copies of $X$ modulo $2$.  If the image in $\zz/2 \zz$ is $1$, then mod $\ell$, the matrix is of the form $\begin{pmatrix} 1 & \Ast \\ 0 & -1 \end{pmatrix}$ (modulo scalars), and hence has distinct eigenvalues mod $\ell$.  Lemma \ref{hensel} guarantees that the same is true mod $\ell^3$, and it is necessarily diagonalizable, hence contained in a Borel.  If a matrix is in the kernel of this map, then it is in the group generated by $M$, and hence contained in a Borel.
\end{rem}

\section{Analysis at $\ell =2$}\label{two}

Although the above analysis in $\S$\ref{sec_gpthy} only holds for odd primes, we may computationally explore the picture at $2$.  The \textsc{Sage} and \textsc{Magma} code necessary to find all exceptional subgroups of $\GL_2(\zz/2^n\zz)$ for $n \leq 6$ can be found at \url{https://github.com/ivogt161/isogeny}.

This classification differs from the case of odd primes $\ell$ in two ways.  There we only classified \emph{minimally} exceptional images of Galois, which were assumed to be contained in a Borel subgroup modulo one lower power of $\ell$.  This ignores (1) exceptional groups which are images of $\ell^r$-isogenous curves and (2) exceptional groups modulo higher powers of $\ell$ whose reduction is contained in a minimal one.  However, if there is an exception, then there is a minimal exception, so this simplification turned out to be harmless since there are finitely many lift-exceptional subgroups modulo every odd prime power.  The presence of genus 0 and 1 modular curves for 2-power levels required a finer analysis, see the table in Appendix \ref{2data}.
 
We summarize the results here.

\begin{prop}\label{2exceptions}
Let $H \subset \GL_2(\zz/2^n\zz)$ be exceptional.  Then
\begin{enumerate}
\item $n = 3$ or \(n \geq 5\).
\item If $n=3, 5$ or $6$, then $H$ is a subgroup of one of the maximal exceptional subgroups listed in the table in Appendix \ref{2data}.
\item If $n\geq 6$, then the corresponding modular curve $X_H$ has genus at least $2$.
\end{enumerate}
\end{prop}
\noindent
This proposition can be verified by running the file \texttt{verify\_Proposition\_5-1.py}.

\section{Boundedness of lift-exceptional primes}\label{boundell}

Recall that a prime $\ell$ is called lift-exceptional for the number field $K$ if there exists an elliptic curve $E/K$ with an $\ell$-isogeny such that $\C(E)$ has an $\ell^n$-isogeny locally almost everywhere, but does not have an $\ell^n$-isogeny globally.

The classification Theorem \ref{upthm} shows that if $\ell$ is lift-exceptional for $K$, then there exists some curve $E/K$ with $G_{E}(\ell^{2m+1}) \subseteq R(\ell^{2m+1})$ for some $m \geq 1$, and hence $G_{E}(\ell^{2m+1})$ is necessarily Borel modulo $\ell^{2m}$ and radical modulo $\ell^{m+1}$.  Using this, we bound lift-exceptional primes $\ell$, depending on the number field alone.

We begin with a bound analogous to the bound of Sutherland given in \cite[Theorem 1]{sutherland}.  Being radical modulo $\ell$ alone puts a condition on the number field $K$ when $\ell \equiv 1 \pmod{4}$.

\begin{prop}\label{sqrt}
If $\ell$ is lift-exceptional for $K$ and $\ell \equiv 1 \pmod{4}$, then $\sqrt{\ell} \in K$.
\end{prop}

\begin{proof}
We will show that the existence of an elliptic curve $E/K$ with $G_E(\ell)$ radical forces $\sqrt{\ell} \in K$.
As $\det \rho_{E, \ell}$ is the cyclotomic character, if $\sqrt{\ell} \not\in K$, then 
$G_E(\ell)$ contains elements whose determinant is quadratic nonresidue in $\ff_\ell^\times$.
However, if $\ell \equiv 1 \pmod{4}$, then every element of the radical subgroup $\begin{pmatrix} \chi & \Ast \\ 0 & \pm \chi \end{pmatrix}$ has determinant $\pm \chi^2$, and hence a square modulo $\ell$.
\end{proof}

This bounds lift-exceptional primes $\ell$ that are $1$ modulo $4$ from above by the discriminant of $K$ over $\qq$.  We can obtain a better bound depending on only the degree of $K$ over $\qq$ using the following Lemma and a large image result of Serre.

\begin{lem}\label{iso_inv}
The property that $G_E(\ell)$ is contained in a radical subgroup is an isogeny invariant.
\end{lem}

\begin{proof}
If $E$ and $E'$ are isogenous by a prime-to-$\ell$ degree isogeny, then $E[\ell] \simeq E'[\ell]$ as Galois modules.

Therefore we can assume that $E$ and $E'$ are $\ell$-isogenous.  Let $\psi$ denote the isogeny, with kernel $C_1$.  We then have an exact sequence of $G_K$-modules
\[0 \to C_1 \to E[\ell] \xrightarrow{\psi} C \to 0, \]
for a quotient $C$ isomorphic to the kernel of the dual isogeny $\psi^\vee$.
Write $\chi$ for the isogeny character of $\psi$ (giving the action of $G_K$ on $C_1$) 
and $\chi'$ for the isogeny character of $\psi^\vee$ (giving the action of $G_K$ on $C$).

Now assume that $G_E(\ell)$ is contained in a radical subgroup.  Then $G_E(\ell)$ fixes a line $L \subset \ff_\ell^2$ and an index (at most) two subgroup $H \subseteq G_E(\ell)$ preserves an isomorphism $L \simeq E[\ell]/L$.  

If $L=C_1$, then we have that $\chi = \pm \chi'$ and so $G_{E'}(\ell)$ is also radical.  Otherwise, $L \simeq C$ as $G_K$-module, and again we see that $\chi = \pm \chi'$, and so $G_{E'}(\ell)$ is  radical.
\end{proof}

\begin{lem}[(Serre, c.f. Lemma 18' of \cite{serrechebotarev} in the case $K = \qq$)]\label{largeorder}
The inertia subgroup at $\lambda$ for every $\lambda$ above $\ell \geq 5$ in $\O_K$ of $\pp G(\ell)$ contains an element of order at least $\frac{\ell -1}{3[K:\qq]}$.
\end{lem}

\begin{proof}
Following the proof of Lemma 18' of \cite{serrechebotarev} naturally gives $\frac{\ell -1}{4[K:\qq]}$, but this can be strengthened slightly.  In fact, in the case of potential good reduction at $\ell$, if $E$ attains good reduction over an extension of ramification index $4$, then over a ramified quadratic extension, a quadratic twist of $E$ has good reduction.  Therefore in that argument, we may assume that $e \leq 3$.
\end{proof}

\begin{cor}\label{radCart_bound}
If $G_{E}(\ell)$ is contained in a radical Cartan subgroup, then
\[\ell \leq 6[K:\qq] +1. \]
\end{cor}
\begin{proof}
Under these hypotheses, $\pp G(\ell)$ would have order at most $2$.  Hence, by the previous Lemma, we have that $2 \geq \frac{\ell-1}{3[K:\qq]}$ and so, $\ell \leq 6[K:\qq] + 1$.  
\end{proof}

\begin{cor}\label{finite_primes}
If $\ell$ is lift-exceptional for $K$, then $\ell \leq 6 [K:\qq] + 1$.
\end{cor}
\begin{proof}
From the classification of lift-exceptional subgroups, there exists an elliptic curve $E/K$ with $G_{E}(\ell^{2m+1}) \subseteq R(\ell^{2m+1})$ for $m\geq 1$.  Therefore $G_{E}(\ell)$ is necessarily radical and $G_{E}(\ell^2)$ is necessarily Borel.  Hence $E$ is isogenous to a curve which has two independent $\ell$-isogenies (factoring the $\ell^2$-isogeny).  But by Lemma \ref{iso_inv}, such a curve is also radical modulo $\ell$.  The result now follows from Corollary \ref{radCart_bound}.
\end{proof}

\section{Finiteness of exceptional $j$-invariants}\label{jinvs}

In this section, we prove the finiteness results in Theorem \ref{mainthm}.  Recall the notation
\[\Sigma(K, N) \colonequals \{ j \in K : \text{$j = j(E/K)$ and $\C(E)$ fails $\bs_N$} \}, \qquad \Lambda(K) \colonequals \bigcup_{(N,70)=1} \Sigma(K,N). \] 
Corollary \ref{finite_primes} guarantees that for any number field $K$, there are only finitely many lift-exceptional primes $\ell$.  Define also the set of lift-exceptional $j$-invariants
\[\Sigma^+(K, \ell) \colonequals \{j \in K : \text{$j = j(E/K)$, $\exists k$ such that $(\ell^k, j)$ is a lift-exceptional pair} \}.\]

\begin{prop}\label{finitej}
For a fixed odd prime $\ell$, $\Sigma^+(K, \ell)$ is finite.
\end{prop}

\begin{proof}
If $(\ell^k, j(E))$ is a lift-exceptional pair, then for some $E'/K$ isogenous to $E$, $G_{E'}(\ell^3) \subseteq R(\ell^3)$ or $\C(E)$ has an $\ell^4$-isogeny.  In both cases, the modular curve $X_{R(\ell^3)}$ and $X_0(\ell^4)$ are genus at least $2$ for $\ell \geq 3$.   
\end{proof}

Having addressed the problem of lift-exceptional pairs, we now apply this to say something about exceptional pairs, where we make no assumption that $E$ has an $\ell$-isogeny.
If $\C(E)$ locally almost everywhere has an $\ell^n$-isogeny, then it may fail to have an $\ell^n$-isogeny globally because it fails to have an $\ell$-isogeny.  This could be the case whenever $\ell$ is an exceptional prime (in the sense of \cite{sutherland}) for the number field $K$.  The following theorem of Anni, building on work of Sutherland, bounds such examples:

\begin{thm}[(Anni, Thm 4.3, Cor 4.5, Thm 5.3 of \cite{anni})]\label{anni_bound}
Let $E/K$ be an elliptic curve with an $\ell$-isogeny locally almost everywhere.  Write $\disc_K$ for the discriminant of $K$ over $\qq$.  Then $E$ has an $\ell$-isogeny globally or $\ell \leq \max(\disc_K, 6[K:\qq]+1)$.  Further, for fixed $\ell$, there are finitely many exceptional pairs $(\ell, j(E))$ unless $\ell=5,7$.
\end{thm}

Using this, we can now prove part \eqref{mainthm_2} of Theorem \ref{mainthm}.  
\begin{proof}[of Theorem \ref{mainthm}\eqref{mainthm_2}]
Fix a number field $K$.  We want to show that there are only finitely many exceptional $j$-invariants as $N$ ranges over all integers coprime to $70$.  Let $S_K := \{3, 11, \cdots \ell_m\}$ be the list of primes at most $\max(\disc_K, 6[K:\qq]+1)$ coprime to $2,5$, and $7$, primes for which there can be infinitely many exceptions, as shown in Theorem \ref{anni_bound} and Proposition \ref{2exceptions}.  Any exceptions must ``come from these primes", i.e. for any exceptional pair $(N, j(E))$, there exists a prime $\ell \in S_K$ and $\ell^n \mid\mid N$ such that $(\ell^n, j(E))$ is exceptional.  We may therefore assume that $N$ is a power of a prime in $S_K$, as all possible exceptional $j(E)$ are already realized with these primes alone.

For any $\ell \neq 5,7$, the set $\Sigma(K, \ell)$ is finite by Anni's result Theorem \ref{anni_bound}.  By Proposition \ref{finitej}, the set $\Sigma^+(K, \ell)$ is also finite for $\ell \neq 2$.  We then have
\[ \Lambda(K) = \bigcup_{i=1}^m \Sigma(K, \ell_i) \cup \Sigma^+(K, \ell_i).\]
is a finite union of finite sets, and is hence also finite.  
\end{proof}

To prove part \eqref{mainthm_1} of Theorem \ref{mainthm} we need to more deeply analyze counterexamples arising from the failure of the local-global principle for $5$- and $7$-isogenies, since there exist number fields where these occur infinitely often.  Recall the following result of Anni and Banwait--Cremona.

\begin{prop}[(Prop 1.3 of \cite{bc} and Prop 3.8 of \cite{anni}).] \hfill

\begin{enumerate}
\item If $(5, j(E))$ is an exceptional pair for $K$ then $\sqrt{5} \in K$ and 
\begin{itemize}
\item $\pp G(5) \simeq D_{4}$, the Klein four group, and $G(5)$ is contained in the normalizer of a split Cartan.
\end{itemize}
\item If $(7, j(E))$ is an exceptional pair for $K$ then $\sqrt{-7} \not\in K$ and
\begin{itemize}
\item $\pp G(7) \simeq D_6$, the Dihedral group of order $6$, and $G(7)$ is contained in the normalizer of a split Cartan,
\item $E$ admits a $7$-isogeny over $K(\sqrt{-7})$.
\end{itemize}
\end{enumerate}
\end{prop}

We can explicitly describe the (maximal) exceptional subgroups in these cases.  Let $\alpha$ be a generator of $\ff_\ell^\times$ for $\ell = 5,7$.  Then
\[ G(\ell) \subseteq H_{\ell, \exc} \colonequals \left\{ \begin{pmatrix} \alpha^i & 0 \\ 0 & \alpha^j \end{pmatrix}, \begin{pmatrix} 0 & \alpha^i   \\  \alpha^j & 0 \end{pmatrix} \right\}_{i \equiv j \pmod{2}} \subseteq \GL_2(\zz/\ell\zz). \]

Define $H_{\ell^n, \exc}$ to be the full preimage of $H_{\ell, \exc}$ in $\GL_2(\zz/\ell^n\zz)$.  Then we have

\begin{prop}\label{vert57}
Let $\ell = 5, 7$.
\begin{enumerate}
\item Every element of $H_{\ell^2, \exc}$ has square discriminant.
\item If $G \subseteq H_{\ell^3, \exc}$ satisfies the property that $\pp G(\ell) = \pp H_{\ell, \exc} = D_4$ or $D_6$ and every element has square discriminant, then $\genus(X_G) \geq 2$.
\end{enumerate}
\end{prop}
\begin{proof}
Part (1) is an easy calculation.  For part (2), if $\ell = 7$ this is also easy.  As $E$ always has an $\ell$-isogeny over $K(\sqrt{-7})$, the pair $(7^3, j(E))$ either gives rise to a $K(\sqrt{-7})$ point on $X_0(7^3)$ or $X_{R(7^3)}$.  Both of these modular curves have genus $\geq 2$.

If $\ell = 5$, we use the following elementary observations from representation theory.  We represent $D_4 = \{e, x, y ,xy\}$ with $x^2 = y^2 = e$.  The Klein four group $D_4$ has character table:

\begin{center}
\begin{tabular}{c | c | c | c | c}
& $e$ & $x$ & $y$ & $xy$ \\ \hline
$A$ & 1 & 1 & 1 &1 \\ \hline
$B$ & 1 & -1 & -1 &1 \\ \hline
$C$ & 1 & 1 & -1 &-1 \\ \hline
$D$ & 1 & -1 & 1 &-1 
\end{tabular}
\end{center}

Let $V:= M_2(\ff_5)$ be the $4$-dimensional vector space of $2 \times 2$ matrices over $\ff_5$.  We have an action of $D_4 = \pp H_{5, \exc}$ on $V$ by conjugation.  Using character theory, it is easy to show that
\[V = A \oplus B \oplus C \oplus D, \]
as a representation of $D_4$.  We explicitly write this as
\[A = \begin{pmatrix} a & 0 \\0 & a \end{pmatrix}, \qquad B = \begin{pmatrix} b & 0 \\0 & -b \end{pmatrix}, \qquad C = \begin{pmatrix} 0 & c \\ c & 0 \end{pmatrix}, \qquad D = \begin{pmatrix} 0 & d \\-d & 0 \end{pmatrix}. \]

Note that we have a lift of $H_{5, \exc}$ to $\GL_2(\zz/5^n\zz)$ for any $n$ by replacing $\alpha$ with $\widetilde{\alpha} \in (\zz/5^n\zz)^\times$ of order $4$ lifting $\alpha$.  We will call this lift $\widetilde{H}_{5, \exc}$.

Using this, we claim that it suffices to show that if
$G \subset \GL_2(\zz/5^3\zz)$ is such that $\pp G(5)$ equals $\pp H_{5, \exc}$ and every element of $G$ has square discriminant, then 
\[G(25) \subseteq I \cdot \tilde{H}_{5, \exc}, \]
where $I \subset \GL_2(\zz/25\zz)$ is a subgroup of $1 + A \oplus B$, or  $1 + A \oplus C$, or $1 + A \oplus D$. Indeed, in these cases an easy Magma computation shows that these groups define modular curves with genus at least 2.

To prove the claim, first notice that every element of $G$ can be written as either
\[ g_1 = \begin{pmatrix} \alpha^i + 5x & 5y \\ 5z & \alpha^j + 5w \end{pmatrix}, \qquad \text{or} \qquad g_2 = \begin{pmatrix} 5x & \alpha^i + 5y \\ \alpha^i + 5z & 5w \end{pmatrix}. \]
We can calculate
\[\Delta(g_1) = (\alpha^i - \alpha^j) \Big( (\alpha^i-\alpha^j) + 2 \cdot 5 (x-w) \Big) + 5^2((x-w)^2 + 4 yz),\]
\[\Delta(g_2) \equiv  4 \alpha^i \alpha^j \pmod{5}.\]
So by Hensel's lemma, $\Delta(g_2)$ is always a square.  And $\Delta(g_1)$ is a square if $i \neq j$.  If $i = j$, then $\Delta(g_1)$ is a square if and only if $\sm{x}{y}{z}{w}$ has square discriminant modulo $5$.  We see that this condition depends only upon $g_1$ mod $25$.  So it suffices to determine which subgroups of $H_{5^2, \exc} \subset \GL_2(\zz/25\zz)$ satisfy this property.  Since all such groups $G(25)$ are presented as,
\[1 \to I \to G(25) \to H_5 \to 1, \]
where $H_5 \subseteq H_{5, \exc}$ is a group such that $\pp H_5  = \pp H_{5, \exc} = D_4$,
it suffices to classify $I$.  Further $I \simeq 1 + 5 \cdot J$ for some \emph{additive} subgroup $J$ of $M_2(\ff_5)$.  The matrix $\sm{x}{y}{z}{w}$ above is the corresponding element of $J$ in this decomposition.

The group $J$ has an action of $H_{5}$ by conjugation, making it a representation of $D_4 = \pp(H_{5})$.  Our claim is that the maximal representations giving rise to groups with the desired properties are $A \oplus B, A \oplus C$, and $A \oplus D$.  We have
\[\Delta \begin{pmatrix} a + b & c+d \\ c-d & a-b \end{pmatrix} = 4(b^2 + c^2 - d^2). \]
If $b, c, d$ are drawn from additive subgroups $V_B, V_C, V_D$ of $\ff_5$, at least two of which are nontrivial, then the above discriminant represents nonsquares.  Hence we have the claim, which completes the proof.\end{proof}

\begin{proof}[of Theorem \ref{mainthm}\eqref{mainthm_1}]
To prove part \eqref{mainthm_1} of Theorem \ref{mainthm} we must combine information at different primes.  Fix a number field $K$ and a positive integer $N$.  Write
\[N = \ell_1^{n_1} \cdot \ell_2^{n_2} \cdots \ell_k^{n_k}. \]
If there exist infinitely many exceptional pairs $(N, j(E))$ for $K$, then for each $\ell_i$, either
\begin{enumerate}[(a)]
\item $\genus(X_0(\ell_i^{n_i})) \leq 1$, or
\item there are infinitely many counterexamples for $\ell_i^{n_i}$-isogenies over $K$,
\end{enumerate}
and for at least one prime $\ell_i \mid N$, we must be in case (b).  Case (a) includes prime powers $2$, $3$, $4$, $5$, $7$, $8$, $9$, $11$, $13$, $16$, $17$, $19$, $25$, $27$, $32$, $49$.  We may further characterize the counterexamples in case (b): by Proposition \ref{finitej}, we are in one of the following two cases: 
\begin{itemize}
\item $\ell = 2$ and  $E/K$ has $G(2^n)$ contained in one of the genus 0 or 1 groups in the table in Appendix \ref{2data}; in this case $n=3$ or $5$. 
\item $\ell = 5, 7$ and $(\ell, j(E))$ is exceptional.  In this case, Proposition \ref{vert57} gives that $n = 1$ or $n=2$.
\end{itemize}
However, there is a further condition: the modular curve parameterizing the level structure mod $N$ (which includes the exceptional level structure in case (b) and the Borel level structure in case (a)) must also have small genus.  This is a much stronger condition, which says that the fiber product
\[ X_{\ell_1^{n_1}} \times_{X(1)} \cdots \times_{X(1)} X_{\ell_k^{n_k}}, \]
where $X_{\ell_i^{n_i}}$ is short-hand for the modular curve parameterizing the given level structure mod $\ell_i^{n_i}$, has small genus.

To find the finite list in Theorem \ref{mainthm}, we compute the genera of the fiber products 
\[X_G \times_{X(1)}  X_H \times_{X(1)} X_0(N), \] 
where $G$ is one of the groups with $n \leq 5$ in the table in Appendix \ref{2data}, and $H$ is one of $H_{5, \exc}$ or $H_{7, \exc}$, and $\genus(X_0(N)) = 0$ or $1$, and the levels of all factors are coprime.  This is achieved by functions in the file \texttt{verify\_Theorem\_1.py} available at \url{https://github.com/ivogt161/isogeny}.  The results are as follows:

\begin{itemize}
\item $X_G \times_{X(1)} X_H$ has genus $>1$ for all $G$ and $H$,
\item $X_G \times_{X(1)} X_0(N)$ has genus $=0,1$ if and only if $G = G_{2147}$ or $G_{2177}$ and $N = 3, 5, 9$,
\item $X_H \times_{X(1)} X_0(N)$ has genus $=0,1$ if and only if $H = H_{5, \exc}$ and $N=2$.
\end{itemize}
From this information one can assemble the list 
\begin{align*}
L &=\{5, 5^2, 7, 7^2\} \cup \{2^3, 2^5\} \cup \{3*2^3, 5*2^3, 9*2^3\} \cup \{2*5, 2*25\}, \\
&= \{5, 7, 8, 10,  24, 25, 32, 40, 49, 50, 72\}. \qedhere
\end{align*}
\end{proof}

\section{Exceptional curves with complex multiplication}\label{cm}

In this section we use the extra $\O$-module structure on $E[N]$ when $E$ has complex multiplication by an order $\O$ in an imaginary quadratic field $F$ to classify exceptional $j$-invariants corresponding to CM curves.  
We write $d_F$ for the discriminant of the quadratic field $F$.

\begin{thm}\label{cm_nicebound}
Let $E/K$ have geometric complex multiplication by $\O$ with $\Frac(\O) = F$.  Let $\O_F$ denote the maximal order in $F$.  If $(N, j(E))$ is exceptional, then there exist relatively prime numbers $A, B$ with $N = A\cdot B$ such that 
\[A \leq (\#\O_F^\times \cdot [KF:H_F])^4 \leq (6[K:\qq])^4, \]
and if $F \subset K$, $E$ has a $B$-isogeny, or if $F \not\subset K$, $B$ factors as $B = \prod_i \ell_i^{n_i}$ such that for all $i$ one of the following holds
\begin{itemize}
\item $n_i=1$ and $\ell_i | d_F$, or
\item $\ell_i$ splits in $F$, $\ell_i \equiv 1 \pmod{4}$ and $K \supset \qq(\sqrt{\ell_i})$, or
\item $\ell_i$ splits in $F$, $\ell_i \equiv 3 \pmod{4}$ and $KF= K(\sqrt{-\ell_i})$, or
\item $\ell_i = 2$ splits in $F$ and 
\begin{itemize}
\item $n_i = 1$ or $2$, or
\item $n_i \geq 3$ and $K \supset \qq(\sqrt{2})$ and $KF = K(\sqrt{-2})$,
\end{itemize}
\end{itemize}
and $\C(E)$ does not have an $\ell_i^{n_i}$-isogeny over $K$ for $\ell_i^{n_i} \mid\mid B$ unless $n_i=1$ and $\ell_i \mid d_F$, or $2$ splits in $F$ and $\ell_i^{n_i} = 2,4$.
\end{thm}

The first step in proving this Theorem is reducing to the case of curves with CM by the maximal order $\O_F$.  That is achieved by the following Lemma.

\begin{lem}[{(\cite[Cor 3]{ehom})}]
Let $E/K$ have CM by $\O \subsetneq \O_F$.  Then there exists an elliptic curve $E'/K$ with CM by $\O_F$ and an isogeny $E \to E'$ defined over $K$.
\end{lem}

From this point forward we will assume that $E$ over $K$ has CM by the maximal order $\O \colonequals \O_F$ in the imaginary quadratic field $F$. 
We have the following tower of field extensions:
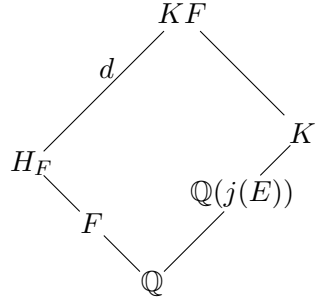
\begin{center}
\begin{tikzpicture}[scale=.8]
\draw (0,0) node{$\qq$};
\draw (-1,1) node{$F$};
\draw (-2,2) node{$H$};
\draw (1.5,1.5) node{$\qq(j(E))$};
\draw (2.5,2.5) node{$K$};
\draw (0.5,4.5) node{$KF$};

\draw (-.2,.2) -- (-.8,.8);
\draw (-1.2, 1.2) -- (-1.8,1.8);
\draw (.2,.2)--(1.2,1.2);
\draw (1.8,1.8) -- (2.3,2.3);

\draw (2.2,2.8) -- (.8,4.2);
\draw (-1.8,2.2) --(.2,4.2);
\draw (-2+1.25, 2+1.25) node[above]{$d$};

\end{tikzpicture}
\end{center}
The field $\qq(j(E))$ is the \defi{field of moduli of $E$} and $H$ denotes the Hilbert class field of $F$.  Recall that $H= F(j(E))$.  We let $d = [KF :H]$ for simplicity.

As a consequence of the main theorem of complex multiplication we have that $E[N]\simeq \O/N \O$ as $\O$-modules and $\rho_{E, N}(G_{H})$ acts on $E[N]$ through $C_{N}(\O) \colonequals \left(\O/N \O\right)^\times$.  By the Chinese remainder theorem we have a canonical isomorphism
\[C_N(\O) \simeq \prod_{\ell^{n} \mid\mid N} C_{\ell^{n}}(\O). \]
It is easy to show \cite[Lemma 2.2]{bourdon} that
\[\#C_{\ell^n}(\O) = \ell^{2n-2}(\ell-1) \left( \ell -  \left(\frac{d_F}{\ell}\right) \right). \]
As in \cite[Theorem 1.1]{bourdon}, for any $E$ defined over $H = F(j(E))$ with CM by $\O$, the reduced Galois representation
\[ \widebar{\rho} \colon G_H \to \frac{\left(\O/N \O\right)^\times}{\O^\times} \]
is surjective.  As a consequence, if $K$ is arbitrary, 
\begin{equation}\label{index_bound}
\prod_{\ell^n \mid\mid N} [C_{\ell^n}(\O): \rho_{E, \ell^n}(G_{KF})] \leq [C_{N}(\O) : \rho_{E, N}(G_{KF}) ] \leq \#\O^\times [KF : H] \leq 6 d. \end{equation}
As as $\zz$-module, we have $\O = \left[ 1,  \left( \frac{ d_F + \sqrt{d_F}}{2} \right) \right].$  Therefore, in this basis, the image of the mod $N$ Galois representation of $E_{KF}$ is contained in matrices of the form
\[ \left\{ \gabold  \right\}, \]
for $a, b \in \zz/N \zz$.

If $E$ is defined over the field of moduli $\qq(j(E))$, then the full group $\rho_N(G_{\qq(j(E))})$ also contains an element corresponding to complex conjugation of $F /\qq$ acting on $\O_F$.  In terms of our chosen basis above, this is of the form
\[c \colonequals \begin{pmatrix} 1 & d_F \\ 0 & -1 \end{pmatrix}. \]
In general, if $F \not\subset K$, then the group $\rho_N(G_{KF}) \subset \rho_N(G_K)$ is an index 2 normal subgroup.  So $\rho_{E, N}(G_K)$ is a subgroup of index less than or equal to $6d$ of the group generated by $C_{N}(\O)$ and $c$, and it does not need to contain $c$.  All elements of $\rho_N(G_K)$ not in $\rho_N(G_{KF})$ are of the form
\[h_{a,b} \colonequals c g_{a,b} =  \habold, \]
for some $a,b$.

The main tool in proving Theorem \ref{cm_nicebound} is the following result for prime powers:

\begin{prop}\label{cmprimepower}
Let $E/K$ have CM by $\O_F$ the full ring of integers of $F$ and let $\ell$ be a prime.  Assume that $\C(E)$ has an $\ell^n$-isogeny locally almost everywhere.

Then either 
\begin{enumerate}[(1)]
\item $n=1$ and $\ell \mid d_F$ and $E$ has an $\ell^n$-isogeny over $K$, or 
\item $\ell$ splits in $F$, and 

\begin{itemize}
\item[$\bullet$] ($F \subset K$): $E$ has an $\ell^n$-isogeny over $K$,
\item[$\bullet$] ($F \not\subset K$): $\ell \equiv 1 \pmod{4}$, $K \supset \qq(\sqrt{\ell})$, and $\C(E)$ does not have an $\ell^n$-isogeny over $K$,
\item[$\bullet$] ($F \not\subset K$): $\ell \equiv 3 \pmod{4}$, $KF = K(\sqrt{-\ell})$, and $\C(E)$ does not have an $\ell^n$-isogeny over $K$,
\item[$\bullet$] ($F \not\subset K$): $\ell^n = 2$ or $4$, and $\C(E)$ has an $2^n$-isogeny over $K$,
\item[$\bullet$] ($F \not\subset K$): $\ell = 2$, $n\geq 3$, $K \supset \qq(\sqrt{2})$ and $KF = K (\sqrt{-2})$, and  $\C(E)$ does not have an $2^n$-isogeny over $K$.
\end{itemize}
\item The index
\begin{align*}  [C_{\ell^n}(\O) : \rho_{E, \ell^n}(G_{KF})] &\geq \begin{cases} \ell^{n/2-1}(\ell-1) &: \ n \geq 4 \text{ even or } \ell \text{ odd and } n=2, \\ \ell^{(n-1)/2} &: \ n\geq 3 \text{ odd} \\ \ell & : \ n=1 \text{ or }\ell=2,n=2. \end{cases}   \\
&\geq \ell^{ n/4 }.
\end{align*}
\end{enumerate}

\end{prop}
\begin{cor}
Let $E/K$ have CM by $\O$ with $\Frac\O = F$ and let $\ell$ be a prime.  If 
\[ \left. \begin{cases}\ell &: \ n=1 \\ \ell^{n/4} &: \ n>1 \end{cases} \right\}  > \# \O_F^\times d, \]
then $(\ell^n, j(E))$ is exceptional if and only if $\ell$ splits in $F$, $F \not \subset K$ and
\begin{itemize}
\item[$\bullet$] $\ell \equiv 1 \pmod{4}$, $K \supset \qq(\sqrt{\ell})$, or
\item[$\bullet$] $\ell \equiv 3 \pmod{4}$, $KF\subset K(\zeta_\ell)$, or
\item[$\bullet$] $\ell = 2$, $n \geq 3$ and both $K \supset \qq(\sqrt{2})$ and $KF = K(\sqrt{-2})$. 
\end{itemize}
\end{cor}

\begin{rem}
In the case $K = \qq$, this gives an alternate proof that for $\ell > 7$, there are no exceptional CM $j$-invariants.
\end{rem}

\begin{proof}[of Proposition \ref{cmprimepower}]
In order for $\C(E)$ to have an $\ell^n$-isogeny locally almost everywhere, every element of $\rho_{E, \ell^n}(G_K)$ must have a root of its characteristic polynomial (equivalently square discriminant when $\ell$ is odd).  We begin by considering the condition this imposes on elements of the index $\leq$ $2$ subgroup $\rho_{E, \ell^n}(G_{KF})$.  Any element has the form
\[ g_{a,b} = \gabold,\]
with characteristic polynomial and discriminant
\[\chi_{a,b}(x) = x^2 + (2a + bd_F)x + a(a+bd_F) - b^2d_F(1-d_F)/4, \qquad \Delta(g_{a,b}) = b^2d_F.\]
Let us split into the following cases:

\textbf{$\ell$ is inert in $F$.}  If $\ell$ is odd, then $d_F$ is a nonzero nonsquare mod $\ell$, and so in order for all $g_{a,b} \in \rho_{E, \ell^n}(G_{KF})$ to have square discriminant, we need $v_\ell(b) \geq \lceil n/2 \rceil$.  Therefore the image of $\rho_{E, \ell^{\lceil \frac{n}{2} \rceil} }(G_{KF})$ is contained in the scalar matrices in $\GL_2(\zz/\ell^{\lceil \frac{n}{2} \rceil} \zz)$.  And so
\begin{align*}
[C_{\ell^n}(\O) : \rho_{E, \ell^n}(G_{KF})]  \geq \frac{\#C_{\ell^{\lceil \frac{n}{2} \rceil} }(\O)}{\#\rho_{E, \ell^{\lceil n/2\rceil}}(G_{KF})} &\geq \frac{\ell^{2\lceil \frac{n}{2} \rceil -2}(\ell-1)(\ell+1) }{\ell^{\lceil \frac{n}{2}\rceil -1}(\ell-1)}, \\
&= \ell^{\lceil \frac{n}{2}\rceil -1}(\ell+1).
\end{align*}
Therefore this falls into case (3).

If $\ell=2$, then $d_F \equiv 5 \pmod{8}$.  Assume that $2 \nmid b$, then 
\[\chi_{a,b}(x) \equiv x^2 + x + 1 \pmod{2}. \]
  
Therefore there are no solutions mod $2$, and hence mod any power of $2$.

If $2 \mid b$, then write $b = 2b'$ for simplicity.  In this case we may change variables to complete the square,
\begin{equation}\label{compsquare}
\chi_{a,b}(x) = (x + (a+b'd_F))^2-(b')^2d_F \end{equation}
 to reduce to the question of whether $(b')^2d_F$ is a square mod $2^n$.  This is impossible if $v_2(b') < \lceil n/2 \rceil -1$, since $d_F$ is not a square mod $8$.  The case $v_2(b') \geq\lceil n/2 \rceil -1$ implies $v_2(b) \geq \lceil n/2 \rceil$ which in turn implies
\[ [C_{2^n}(\O) : \rho_{E, 2^n}(G_{KF})] \geq  \frac{2^{2\left\lceil n/2 \right\rceil-2}\cdot 3}{2^{\left\lceil n/2 \right\rceil-1}} = 3 \cdot 2^{\left\lceil n/2 \right\rceil-1} \geq 2^{\left\lceil n/2 \right\rceil}, \]
and so also falls into case (3).

\textbf{$\ell$ ramifies in $F$.} This is the case when $\ell \mid d_F$.  If $\ell$ is odd, then $v_\ell(d_F) = 1$ and so we need $v_\ell(b) \geq \lfloor n/2 \rfloor$ for all $g_{a,b} \in \rho_{E, \ell^n}(G_{KF})$.  As above this implies that
\[[C_{\ell^n}(\O) : \rho_{E, \ell^n}(G_{KF})]  \geq [C_{\ell^{\lfloor n/2 \rfloor}}(\O) : \rho_{E, \ell^{\lfloor n/2 \rfloor}}(G_{KF})]  \geq \ell^{\lfloor n/2 \rfloor}, \]
and so falls into case (3) if $n \geq 2$.  If $n=1$ then as $d_F$ is zero mod $\ell$, then $\Delta(g_{a,b})$ is trivially always a square.  In addition $\sv{0}{1}$ is a simultaneous eigenvector  of all $g_{a,b} \in \rho_{E, \ell}(G_{KF})$ as well as $c$; hence it is necessarily a simultaneous eigenvector of all elements of $\rho_{E, \ell}(G_K)$.  Therefore $E/K$ always has an $\ell$-isogeny.  For a discussion of this isogeny, see \cite[Sections 12 and 13]{gross}.

If $\ell=2$, then $v_2(d_F) = 2$ or $3$, corresponding to whether $d_F/4$ is $3$ or $2$ mod $4$, respectively.  We change variables to complete the square and the characteristic polynomial of $g_{a,b}$ simplifies to
\[(x-(a+bd_F/2))^2 - b^2 d_F/4. \]
This has a solution if and only if $b^2 d_F/4$ is a square mod $2^n$. 

If $d_F/4 \equiv 2 \pmod 4$, then it has odd $2$-adic valuation and hence $b^2d_F/4$ does as well. So it is a square if and only if $v_2(b) \geq \left\lfloor n/2 \right\rfloor$.  Exactly as in the odd case, this implies that 
\[ [C_{2^n}(\O) : \rho_{E, 2^n}(G_{KF})] \geq 2^{\lfloor n/2 \rfloor}, \]
and so falls into case (3) unless $n=1$.  In that case $\sv{0}{1}$ is again a simultaneous eigenvector of all of $\rho_2(G_K)$ and so $E$ has a $2$-isogeny over $K$.

If $d_F/4 \equiv 3 \pmod 4$ and $2 \nmid b$, then $b^2d_F/4$ is a unit, which is not a square mod 4.  If $v_2(b) < \lfloor n/2 \rfloor$, then $b^2d_F/4$ is not a square mod $2^n$.  If $v_2(b) \geq \lfloor n/2 \rfloor$, then we are in case (3) unless $n=1$, in which case we are in case (1).  One of $a$ or $b$ is always $0$ mod $2$.  Therefore $\sv{1}{1}$ is a simultaneous eigenvector of all of $\rho_{E,2}(G_{KF})$ and also $c$, and therefore all of $\rho_{E,2}(G_K)$.

\textbf{$\ell$ splits in $F$.}  If $\ell$ is odd, then $d_F$ is a nonzero square mod $\ell$ and so we have $D \in (\zz/\ell^n \zz)^\times$ such that $D^2 \equiv d_F \pmod{\ell^n}$.  If $\ell =2$, and $2 \nmid b$, then the characteristic polynomial of $g_{a,b}$ reduces to $x^2 +x$ mod $2$.  By Hensel's lemma this has distinct roots mod all powers of $2$.  If $2 \mid b$, we again let $b= 2b'$ and complete the square as in \eqref{compsquare} to get that $(b')^2d_F$ must be a square mod $2^n$.  But this is always a square mod all powers of $2$ as $d_F$ is a unit square mod $8$.  Hence for all primes $\ell$ splitting in $F$, all elements of $\rho_{E, \ell^n}(G_{KF})$ have a rational root of their characteristic polynomial.  Let us now determine when the same is true for $\rho_{E, \ell^n}(G_K)$ in the case that $F \not\subset K$.

Every  $h_{a,b} \in \rho_{E, \ell^n}(G_K) \smallsetminus \rho_{E, \ell^n}(G_{KF})$ is of the form 
\[h_{a,b} = cg_{a,b} = \habold.\]
Since this is trace $0$, the characteristic polynomial has a root if and only if $-\det = a^2 + abd_F - b^2d_F(1-d_F)/4$ is a square.   We split into the following cases:

If $\ell \equiv 1 \pmod{4}$, then using multiplicativity of determinants, the discriminant of $g \in \rho_{E, \ell^n}(G_K)$ is always a square if and only if the determinant is always a square.  Using the Weil pairing this occurs if and only if $K \supset \qq(\sqrt{\ell})$.  

If $\ell \equiv 3 \pmod{4}$, then the discriminant of $g \in \rho_{E, \ell^n}(G_K))$ is always a square if and only if for all $g\in \rho_{E, \ell^n}(G_{KF})$ we have that $\det(g)$ is a square and for all $h \in \rho_{E, \ell^n}(G_K) \smallsetminus \rho_{E, \ell^n}(G_{KF})$ we have that $\det(h)$ is not a square.  Hence the normalizer character for $(KF) / K$ factors through the determinant character, which is the cyclotomic character.  In fact it must factor through the quadratic character of the unique quadratic subfield of $K(\zeta_\ell)$, so equivalently,
\[ KF = K(\sqrt{-\ell}). \]
In particular, note that $\ell$ ramifies in $K$, since it splits in $F$.

If $\ell = 2$ and $n\geq 3$, then by Hensel's Lemma {\cite[Corollary 7.3]{lang_alg}} the negative of the determinant of $h_{a,b}$ is a square if and only if it is a square mod $8$.  Therefore every element of $\rho_{E, 2^n}(G_K)$ has square discriminant if and only if the determinant restricts to $+1 \mod 8$ on $\rho_{E, 2^n}(G_{KF})$ and $-1 \mod 8$ on $\rho_{E, 2^n}(G_K) \smallsetminus \rho_{E, 2^n}(G_{KF})$.  This implies first that $K$ contains the quadratic subfield of $\qq(\zeta_8)$ determines by $\{\pm1\} \subset (\zz/8\zz)^\times$, namely $\qq(\sqrt{2})$.  Furthermore, the normalizer character for $KF /K$ must factor through the cyclotomic character.  Therefore
\[KF = K(\zeta_8) = K(\sqrt{-2}), \]
as $K$ already contains $\sqrt{2}$.

This completes the forward implication of Proposition \ref{cmprimepower}.  All that remains is to show that for split prime powers large enough to be excluded from case (3), $\ell^n$ is exceptional if and only if $F \not \subset K$ and $\ell^n \neq 2, 2^2$.

The elements $(\pm D + \sqrt{d_F})/2 \in \O_F$ represented by the vectors
\[v_{\pm} \colonequals \begin{pmatrix} (-d_F\pm D)/2 \\ 1 \end{pmatrix}, \]
are simultaneaous eigenvectors of all of $\rho_{E, \ell^n}(G_{KF})$.  Therefore $E_{KF}$ has an $\ell^n$-isogeny.  So we may assume that $F \not\subset K$.   Changing into the $[v_+, v_-]$ basis, everything is of the form 
\[\tilde{g}_{a,b} \colonequals \begin{pmatrix} a + b \left( \frac{d_F +D}{2}\right) & 0 \\ 0 &a + b \left( \frac{d_F -D}{2}\right) \end{pmatrix} ,\] 
and complex conjugation simply swaps $v_+$ and $v_-$.  This is not exceptional modulo $2$ and $2^2$ since it reduces to a scalar mod $2$.

Define
\[r \colonequals \min_{g_{a,b} \in \rho_{E, \ell^n}(G_{KF})}(v_\ell(b)).\]
The only simultaneous eigenvectors of all $g_{a.b}$ mod $\ell^k$ are congruent to $v_{\pm}$ mod $\ell^{k-r}$, and hence in the $v_{\pm}$ basis, are given by $\sv{\ell^{k-r}x}{1}$, $\sv{1}{\ell^{k-r}x}$ for any choice of $x$.  In order for a vector congruent to $v_{\pm}$ mod $\ell^{k-r}$ to be an eigenvector of a $\tilde{h}_{a,b} \colonequals \tilde{c} \tilde{g}_{a,b} = \sm{0}{a+b(d_F +D)/2}{a+b(d_F-D)/2}{0}$ mod $\ell^k$ we must have
\[ \ell^{2(k-r)}x^2(a \mp bD) \equiv a\pm bD \pmod{\ell^k}. \]
But by assumption $a\pm bD$ is a unit in $\O/\ell^n \O$, and hence we must have $k=r$.  Therefore there are no common eigenvectors for the entire normalizer (e.g. for the  $\tilde{g}_{a,b}$ and the $\tilde{h}_{a', b'}$) modulo any larger powers of $\ell$ than $\ell^r$.

If we assume that \(r < \lfloor n/2 \rfloor\), then there do not exist \(0 \leq a \leq c \leq b\) with \(b + c = n\) such that \(\rho_{E, \ell^n}(G_K)\) is conjugate to a subgroup of \(A_{a,c,b}(\ell^n)\), since \(a + b < \lfloor n/2 \rfloor\) implies \(b + c < n\).
Therefore $\C(E_K)$ does not have an $\ell^n$-isogeny.  As above, modulo $\ell^r$, $\rho_{E, \ell^n}(G_{KF})$ is contained in the scalar matrices, so
\[[C_{\ell^n}(\O) : \rho_{E, \ell^n}(G_{KF})] \geq \frac{\ell^{2r-2} (\ell-1)^2}{\ell^{r-1}(\ell-1)} = \ell^{r-1}(\ell-1). \]
So $r \geq  \lfloor n/2 \rfloor$ is in case (3).

Finally, $\ell^{n/4}$ is a lower bound for the more refined bounds in part (3).\end{proof}

\begin{proof}[of Theorem \ref{cm_nicebound}]
Factor $N$ as $\prod_{i \in S} \ell_i^{n_i}$.  By assumption of $\C(E)$ having an $N$-isogeny locally almost everywhere, $\C(E)$ has an $\ell_i^{n_i}$-isogeny locally almost everywhere for each $i$.

We will define $B$ as a subproduct of those prime powers such that $\ell_i^{n_i} > [C_{\ell_i^{n_i}}(\O):\rho_{\ell_i^{n_i}}(G_{KF})]^4$ and
\begin{enumerate}
\item If $F \subset K$:
\begin{itemize}
\item $n_i=1$ and $\ell_i | d_F$
\item $\ell_i$ splits in $F$.
\end{itemize}
\item If $F \not\subset K$:
\begin{itemize}
\item $n_i=1$ and $\ell_i | d_F$
\item $\ell_i$ splits in $F$, $\ell_i \equiv 1 \pmod{4}$ and $K \supset \qq(\sqrt{\ell_i})$,
\item $\ell_i$ splits in $F$, $\ell_i \equiv 3 \pmod{4}$ and $KF= K(\sqrt{-\ell_i})$
\item $\ell _i= 2$ splits in $F$ and either $n_i=1,2$ or $K \supset \qq(\sqrt{2})$ and $KF = K(\sqrt{-2})$.
\end{itemize}
\end{enumerate}
Proposition \ref{cmprimepower} implies that $E$ has a $\ell_i^{n_i}$-isogeny for each $\ell_i^{n_i} \mid\mid B$ if and only if $F \subset K$ or $n_i=1$ and $\ell_i \mid d_F$, or $\ell_i=2$ and $n_i \leq 2$.

What remains is to show that the quotient $A \colonequals N/B$ is small.  Let $A = \prod_j p_j^{m_j}$.  By \eqref{index_bound} and Proposition \ref{cmprimepower} 
\begin{align*}
6d &\geq  [C_A(\O):\rho_{E,A}(G_{KF})] \\
& \geq  \prod_j [C_{p_j^{m_j}}(\O) : \rho_{E, p_j^{m_j}}(G_{KF})],
\intertext{and each prime $p_j$ not in $B$ must fall into case (3) of Proposition \ref{cmprimepower}, so}
&\geq \prod_j p_j^{m_j/4} = A^{1/4},
\end{align*}
and the result follows.
\end{proof}

\section{Exceptional primes for $K = \qq$}\label{qpoints}

From \cite[Thm 2]{sutherland}, the only exceptional pair for prime $N$ over $\qq$ is $(7, 2268945/128)$.  Hence for all other primes $\ell$, if $E$ locally almost everywhere has an $\ell$-isogeny, then it has an $\ell$-isogeny over $\qq$.  Since our goal here is to find all prime power exceptions over $\qq$, we can use Theorem \ref{upthm} for odd prime powers and Proposition \ref{2exceptions} for powers of $2$.

By considering $X_0(\ell^n)(\qq)$, it has been shown that there exist $\ell^n$-isogenies over $\qq$ if and only if $\ell^n = 2,3,4,5,7,8,9,13,16,25,27,37,43,67,163$ (see table and Theorem 1 in \cite{kenku}).  In particular, there are no $\ell^2$-isogenies over $\qq$ for $\ell \geq 7$.  As $5 \equiv 1 \pmod{4}$, Proposition \ref{sqrt} guarantees that the only exceptional pairs over $\qq$ could come from $\ell = 2,3,7$.

\subsection{Exceptional subgroups at $\ell = 3$}

$X_0(3^r)(\qq)$ has only cuspidal rational points when $r \geq 4$.  As any exceptional subgroup modulo $3^r$ must be lift-exceptional, all exceptional $j$-invariants must give rise to rational points on the modular curve $X_{R(27)}$, as we now show.  Indeed, by Theorem \ref{upthm}, if $(3^r, j(E))$ is lift-exceptional, then either $j(E') \in X_{R(27)}(\qq)$ or $j(E') \in X_0(3^4)(\qq)$, for $E'$ $\qq$-isogenous to $E$.  But $X_0(3^4)$ has no noncuspidal rational points.

The group $R(27)$ corresponds to the congruence subgroup 27B${}^4$ in the Cummins--Pauli database \cite{cp}, and as such $X_{R(27)}$ is genus 4.  The group $R(27)$ is not ``arithmetically maximal" (as in Definition 3.1 of \cite{rzb}).  It is conjugate to a subgroup of the genus $2$ group $G_{32} \subset \GL_2(\zz/3^3\zz)$ with generators
\[ \begin{pmatrix}26 & 0 \\ 0 & 26 \end{pmatrix}, \begin{pmatrix}11 & 21 \\ 0 & 5\end{pmatrix}, \begin{pmatrix}19 & 9 \\ 0 &10 \end{pmatrix}, \begin{pmatrix}16 &  3 \\ 0 &  1\end{pmatrix}, \begin{pmatrix}8 & 18 \\ 0 & 26 \end{pmatrix}, \begin{pmatrix} 1 & 0 \\ 0 & 26 \end{pmatrix}, \begin{pmatrix}1 & 6 \\ 0 & 1 \end{pmatrix}, \begin{pmatrix}10 &  2 \\ 18 &  1\end{pmatrix}, \begin{pmatrix}1 & 9 \\ 0 & 1\end{pmatrix}. \]
Using the publically available Magma code from \cite{rzb} in the case $\ell=3$, Jeremy Rouse and David Zureick-Brown computed that the corresponding modular curve $X_{32}$ has equation
\[X_{32} : y^2 + (x^3+z^3)y = -5x^3z^3 - 7z^6, \]
in the weighted projective space $\pp(1,3,1)$.  They also computed that the map $X_{32} \to X(1)$ is given by
{\footnotesize
\[j = \frac{ -(-3x^9 + 99x^6y - 189x^6z^3 + 63x^3y^2 + 126x^3yz^3 - 441x^3z^6 + 25y^3 - 21y^2z^3 + 147yz^6 - 343z^{12})^3(-3x^3 + y - 7z^3)^3}{(4y - 7z^3)(x^3 + y)^9(27x^6 + 18x^3y + 63x^3z^3 + 7y^2 + 7yz^3 + 49z^6)}.\]
}
The Jacobian of $X_{32}$ is rank $0$, so the implementation of Chaubaty's method in Magma gives that
\[ X_{32}(\qq) = \{ [1:-1:0], [1:0:0]\}. \]
Using the $j$-map above, both rational points of $X_{32}$ are cuspidal, and so the same must be true of any rational points on $X_{R(27)}$.

\subsection{Exceptional subgroups at $\ell = 7$}

Since there are no elliptic curves over $\qq$ with $7^2$-isogenies, all exceptions $7$-adically over $\qq$ come from exceptions mod $7$.  Thus $j = 2268945/128$ is the only $j$-invariant giving $7$-adic exceptions over $\qq$.  From Proposition \ref{vert57} we have that this $j$-invariant is also exceptional mod $49$. However, it is not exceptional mod $7^3$; by sampling Frobenius elements, one quickly sees that every element of $\rho_{E, 7^3}(G_\qq)$ need not have square discriminant.  For example $\#E(\ff_{53}) = 58$, so $a_{53} = -4$ and
\[\Delta(\rho(\Frob_{53})) \equiv -2^27^2 \pmod{7^3}, \]
which is not a square as $7 \equiv 3 \pmod{4}$.

\subsection{Exceptional subgroups at $\ell = 2$}

\begin{table}[ht]
\begin{tabular}{c | c c c c}
label & $n$ & RZB label & RZB cover& $\qq$-points on RZB cover \\ \hline
%27445 & 4 & 92 & 92 & family $j_1$ \\
189621 & 5 & 353 & 353 &cuspidal \\
%189900 & 5 & 305 & 168 &cuspidal \\
189995 & 5 & 314 & 158 & cuspidal + (j=287496 (CM by -16)) \\
190435 & 5 & 278 & 54 & cuspidal +   (j=1728)\\
190525 & 5 & 255 &54 & cuspidal + (j=1728)\\
890995& 6 & 667 &354 & cuspidal\\
891525& 6 & 627 & 168 & cuspidal\\
891526& 6 & 617 &159 &cuspidal\\
891735& 6 & 636 &168 &cuspidal\\
891737& 6 & 621 &159 &cuspidal \\
%891738& 6 & 638 & 168 &cuspidal\\
893009& 6 & 612 &52 &cuspidal \\
893011& 6 & 614 &51 &cuspidal +  (j=8000 (CM by -8)\\
893326& 6 & 603 &54  &cuspidal +  (j=1728)\\
%893327& 6 & 544 &53 &cuspidal +   (j=1728)\\
894711& 6 & 541 &51 &cuspidal +   (j=8000 (CM by -8))\\
\end{tabular}
\caption{Rational points on 2-adic exceptional modular curves.}
\end{table}

Proposition \ref{2exceptions} gives all maximal exceptional subgroups of $\GL_2(\zz/2^n\zz)$ for $n \leq 6$ (which in particular covers all exceptions over $\qq$ since $X_0(32)$ contains no noncuspidal rational points).  In order for $G\subset \GL_2(\zz/2^n\zz)$ to occur over $\qq$, the determinant map $\det \colon G \to (\zz/2^n \zz)^\times$ must be surjective.  Table \ref{2data} gives labels from the recent work of Sutherland--Zywina \cite{sz} and Rouse--Zureick-Brown \cite{rzb} for such groups which can occur over $\qq$.  The relevant information from \cite{rzb} is contained in Table 1.  

As above we rule out the CM $j$-invariants not $1728$ by sampling discriminants of Frobenius elements.  For any $E/\qq$ with $j(E) = 1728$, and any rational prime $p$ which is inert in $\qq(i)$, the Frobenius element $\Frob_{p}$ is contained in the complement of the Cartan subgroup $\rho_{E, 2^5}(G_{\qq(i)})$.  As in $\S$\ref{cm}, such matrices have trace $0$, and hence the discriminant depends only upon the determinant.  Hence $\Delta(\rho_{E, 2^5}(\Frob_p))$ is a square modulo $2^5$ if and only if $-4\cdot p$ is a square modulo $2^5$, which is visibly independent of twists.  As $-11 \equiv 5 \pmod{8}$ is not a square, $\Frob_{11}$ witnesses the fact that $j=1728$ is not exceptional.

\appendix
\section{Maximal exceptional subgroups of $\GL_2(\zz/2^n\zz)$ for $n\leq 6$}\label{2data}

This table lists maximal exceptional subgroups of $\GL_2(\zz/2^n\zz)$ for $n \leq 6$.  The SZ label corresponds to the paper \protect\cite{sz} and the RZB label corresponds to the transpose group in the paper \protect\cite{rzb}.  As both papers pertain only to groups arising from curves over $\qq$, if the determinant is not surjective we do not give a label.

{\scriptsize
\begin{tabular}{c c c c c c c}
label & $n$& level & genus & generators &SZ & RZB \\ \hline
2147& 3 & 4 & 0 & $\sm{1}{1}{0}{1},\sm{3}{0}{0}{3},\sm{7}{4}{4}{3},\sm{5}{4}{4}{5}$ &  \multicolumn{2}{c}{(det not surjective)} \\[5pt]
2177& 3 & 4 & 0 & $\sm{1}{2}{0}{1},\sm{7}{4}{4}{3},\sm{3}{4}{4}{3},\sm{7}{6}{2}{3}$ &  \multicolumn{2}{c}{(det not surjective)} \\[5pt]
%27445& 4 & 8 & 0 & $\sm{1}{2}{0}{1},\sm{3}{1}{0}{1},\sm{3}{0}{0}{3},\sm{1}{8}{8}{9},\sm{13}{8}{8}{5},\sm{3}{9}{0}{5}$ & 8G0-8f & 92\\[5pt]
189551& 5 & 16 & 0 & $\sm{1}{1}{0}{1},\sm{3}{0}{0}{1},\sm{3}{0}{0}{3},\sm{1}{16}{16}{17},\sm{29}{16}{16}{13}$ &  \multicolumn{2}{c}{(det not surjective)} \\[5pt]
189605& 5 & 16 & 0 & $\sm{1}{1}{0}{1},\sm{9}{0}{0}{1},\sm{3}{0}{0}{3},\sm{1}{16}{16}{17},\sm{29}{16}{16}{13},\sm{3}{9}{16}{21}$ &  \multicolumn{2}{c}{(det not surjective)} \\[5pt]
189621& 5 & 32 & 1 & $\sm{1}{1}{0}{1},\sm{3}{0}{0}{3},\sm{21}{0}{0}{21},\sm{3}{0}{8}{9},\sm{15}{27}{8}{9}$ & 32A1-32b & 353\\[5pt]
189785& 5 & 8 & 0 & $\sm{1}{2}{0}{1},\sm{3}{0}{0}{1},\sm{3}{0}{0}{3},\sm{1}{24}{8}{25},\sm{13}{24}{24}{21}$ &  \multicolumn{2}{c}{(det not surjective)} \\[5pt]
189892& 5 & 8 & 0 & $\sm{1}{2}{0}{1},\sm{3}{1}{0}{1},\sm{3}{0}{0}{3},\sm{1}{24}{8}{25},\sm{13}{24}{24}{21}$ &  \multicolumn{2}{c}{(det not surjective)} \\[5pt]
%189900& 5 & 16 & 1 & $\sm{1}{2}{0}{1},\sm{3}{1}{0}{1},\sm{3}{0}{0}{3},\sm{1}{16}{16}{17},\sm{29}{16}{16}{13},\sm{3}{9}{16}{21}$ & 16E1-16h & 305\\[5pt]
189979& 5 & 8 & 0 & $\sm{1}{2}{0}{1},\sm{9}{0}{0}{1},\sm{3}{0}{0}{3},\sm{1}{24}{8}{25},\sm{13}{24}{24}{21},\sm{13}{30}{4}{3}$ &  \multicolumn{2}{c}{(det not surjective)} \\[5pt]
189981& 5 & 8 & 0 & $\sm{1}{2}{0}{1},\sm{9}{0}{0}{1},\sm{3}{0}{0}{3},\sm{1}{24}{8}{25},\sm{13}{24}{24}{21},\sm{3}{9}{16}{21}$ &  \multicolumn{2}{c}{(det not surjective)} \\[5pt]
189995& 5 & 16 & 1 & $\sm{1}{2}{0}{1},\sm{3}{0}{0}{3},\sm{29}{16}{16}{13},\sm{3}{0}{4}{9},\sm{7}{4}{4}{9}$ & 16E1-16b & 314\\[5pt]
190318& 5 & 8 & 0 & $\sm{1}{4}{0}{1},\sm{3}{0}{0}{1},\sm{3}{0}{0}{3},\sm{29}{28}{20}{21},\sm{19}{12}{4}{3}$ &  \multicolumn{2}{c}{(det not surjective)} \\[5pt]
190435& 5 & 8 & 1 & $\sm{3}{1}{0}{1},\sm{3}{0}{0}{3},\sm{1}{24}{8}{25},\sm{13}{24}{24}{21},\sm{3}{9}{16}{21}$ & 8F1-8k & 278\\[5pt]
190487& 5 & 8 & 0 & $\sm{1}{4}{0}{1},\sm{3}{0}{0}{3},\sm{1}{24}{8}{25},\sm{29}{28}{20}{21},\sm{19}{12}{4}{3},\sm{13}{30}{4}{3}$ &  \multicolumn{2}{c}{(det not surjective)} \\[5pt]
190525& 5 & 8 & 1 & $\sm{1}{4}{0}{1},\sm{3}{0}{0}{3},\sm{13}{24}{24}{21},\sm{9}{0}{2}{3},\sm{13}{12}{10}{3}$ & 8F1-8j & 255\\[5pt]
876594& 6 & 64 & 3 & $\sm{1}{1}{0}{1},\sm{9}{0}{0}{1},\sm{3}{0}{0}{3},\sm{5}{0}{0}{5},\sm{15}{27}{8}{9}$ &  \multicolumn{2}{c}{(det not surjective)} \\[5pt]
878116& 6 & 32 & 3 & $\sm{1}{2}{0}{1},\sm{9}{0}{0}{1},\sm{3}{0}{0}{3},\sm{21}{32}{32}{53},\sm{13}{30}{4}{3}$ &  \multicolumn{2}{c}{(det not surjective)} \\[5pt]
881772& 6 & 16 & 3 & $\sm{1}{4}{0}{1},\sm{9}{0}{0}{1},\sm{3}{0}{0}{3},\sm{29}{16}{48}{45},\sm{13}{12}{10}{3}$ &  \multicolumn{2}{c}{(det not surjective)} \\[5pt]
885865& 6 & 8 & 3 & $\sm{1}{8}{0}{1},\sm{3}{0}{0}{3},\sm{33}{56}{40}{57},\sm{0}{9}{1}{0},\sm{45}{24}{24}{21}$ &  \multicolumn{2}{c}{(det not surjective)} \\[5pt]
890995& 6 & 64 & 3 & $\sm{1}{1}{0}{1},\sm{3}{0}{0}{1},\sm{3}{0}{0}{3},\sm{5}{0}{0}{5},\sm{3}{9}{32}{15}$ &  & 667\\[5pt]
891525& 6 & 32 & 3 & $\sm{1}{2}{0}{1},\sm{3}{0}{0}{1},\sm{3}{0}{0}{3},\sm{21}{32}{32}{53},\sm{19}{30}{16}{21}$ &  & 627\\[5pt]
891526& 6 & 32 & 3 & $\sm{3}{0}{0}{1},\sm{3}{0}{0}{3},\sm{1}{32}{32}{33},\sm{21}{32}{32}{53},\sm{3}{9}{32}{15}$ &  & 617\\[5pt]
891735& 6 & 32 & 3 & $\sm{1}{2}{0}{1},\sm{3}{1}{0}{1},\sm{3}{0}{0}{3},\sm{21}{32}{32}{53},\sm{3}{9}{16}{21}$ &  & 636\\[5pt]
891737& 6 & 32 & 3 & $\sm{3}{1}{0}{1},\sm{3}{0}{0}{3},\sm{1}{32}{32}{33},\sm{21}{32}{32}{53},\sm{3}{9}{32}{15}$ &  & 621\\[5pt]
%891738& 6 & 32 & 3 & $\sm{1}{2}{0}{1},\sm{3}{1}{0}{1},\sm{3}{0}{0}{3},\sm{1}{32}{32}{33},\sm{21}{32}{32}{53},\sm{35}{24}{32}{15}$ &  & 638\\[5pt]
893009& 6 & 16 & 3 & $\sm{3}{0}{0}{1},\sm{3}{0}{0}{3},\sm{1}{48}{16}{49},\sm{29}{16}{48}{45},\sm{19}{30}{16}{21}$ &  & 612\\[5pt]
893011& 6 & 16 & 3 & $\sm{1}{4}{0}{1},\sm{3}{0}{0}{1},\sm{3}{0}{0}{3},\sm{29}{16}{48}{45},\sm{23}{36}{8}{9}$ &  & 614\\[5pt]
893326& 6 & 16 & 3 & $\sm{3}{1}{0}{1},\sm{3}{0}{0}{3},\sm{1}{48}{16}{49},\sm{29}{16}{48}{45},\sm{3}{9}{16}{21}$ &  & 603\\[5pt]
%893327& 6 & 16 & 3 & $\sm{9}{0}{0}{1},\sm{3}{1}{0}{1},\sm{3}{0}{0}{3},\sm{1}{48}{16}{49},\sm{29}{16}{48}{45},\sm{35}{51}{16}{21}$ &  & 544\\[5pt]
894711& 6 & 8 & 3 & $\sm{3}{0}{0}{1},\sm{3}{0}{0}{3},\sm{33}{56}{40}{57},\sm{45}{24}{24}{21},\sm{23}{24}{4}{3}$ &  & 541\\[5pt]
\end{tabular}
}

\subsection*{Acknowledgements}\label{ackref}
I would like to thank  Andrew Sutherland for useful conversations, suggestions, and data on subgroups of $\GL_2(\zz/\ell^n\zz)$ for small $\ell$ and $n$.  I would also like to thank  David Zureick-Brown and Jeremy Rouse for help with finding explicit equations for relevant modular curves.   
Thank you to Ariel Weiss and Amit Ophir for pointing out that the ``equivalently all'' characterization in Lemma 3.2 in the published version of this paper is incorrect when the isogeny class contains curves with full (projective) level \(\ell\) structure and for drawing my attention to the fact that the case \(K(\ell^n)\) contained in a split Cartan mod \(\ell\) required special attention.
In addition, many thanks to Alina Cojocaru, Noam Elkies, Dick Gross, Nathan Jones, Eric Larson, Bjorn Poonen, Barry Mazur, and Dmitry Vaintrob for comments and discussions.  Finally I would like to warmly thank the anonymous referee for a careful reading of the manuscript and many suggestions that have improved the paper.

\providecommand{\bysame}{\leavevmode\hbox to3em{\hrulefill}\thinspace}
\providecommand{\MR}{\relax\ifhmode\unskip\space\fi MR }
\providecommand{\MRhref}[2]{%
  \href{http://www.ams.org/mathscinet-getitem?mr=#1}{#2}
}
\providecommand{\href}[2]{#2}

%\affiliationone{% in this example, two authors share an institution
%   Isabel Vogt\\
%   Stanford Department of Mathematics \\
%   450 Jane Stanford Way, Building 380 \\
%   Stanford, CA 94305-2125\\
%   USA
%   \email{ivogt.math@gmail.com}}


\begin{thebibliography}{99}% Replace 9 by 99 if 10 or more references
\bibitem{anni}
Samuele Anni, \emph{A local-global principle for isogenies of prime degree over
  number fields}, J. Lond. Math. Soc. (2) \textbf{89} (2014), no.~3, 745--761.
  \MR{3217647}

\bibitem{bc}
Barinder~S. Banwait and John~E. Cremona, \emph{Tetrahedral elliptic curves and
  the local-global principle for isogenies}, Algebra Number Theory \textbf{8}
  (2014), no.~5, 1201--1229. \MR{3263141}

\bibitem{bourdon}
Abbey Bourdon and Pete~L. Clark, \emph{Torsion points and {G}alois
  representations on {CM} elliptic curves}, Pacific J. Math. \textbf{305}
  (2020), no.~1, 43--88. \MR{4077686}

\bibitem{cp}
C.~J. Cummins and S.~Pauli, \emph{Congruence subgroups of {${\rm PSL}(2,{\Bbb
  Z})$} of genus less than or equal to 24}, Experiment. Math. \textbf{12}
  (2003), no.~2, 243--255. \MR{2016709}

\bibitem{gross}
Benedict~H. Gross, \emph{Arithmetic on elliptic curves with complex
  multiplication}, Lecture Notes in Mathematics, vol. 776, Springer, Berlin,
  1980, With an appendix by B. Mazur. \MR{563921 (81f:10041)}

\bibitem{katz}
Nicholas~M. Katz, \emph{Galois properties of torsion points on abelian
  varieties}, Invent. Math. \textbf{62} (1981), no.~3, 481--502. \MR{604840
  (82d:14025)}

\bibitem{kenku}
M.~A. Kenku, \emph{On the modular curves {$X_{0}(125)$}, {$X_{1}(25)$}\ and
  {$X_{1}(49)$}}, J. London Math. Soc. (2) \textbf{23} (1981), no.~3, 415--427.
  \MR{616546}

\bibitem{lang_alg}
Serge Lang, \emph{Algebra}, third ed., Graduate Texts in Mathematics, vol. 211,
  Springer-Verlag, New York, 2002. \MR{1878556}

\bibitem{rzb}
Jeremy Rouse and David Zureick-Brown, \emph{Elliptic curves over {$\Bbb Q$} and
  2-adic images of {G}alois}, Res. Number Theory \textbf{1} (2015), Art. 12,
  34. \MR{3500996}

\bibitem{serrechebotarev}
Jean-Pierre Serre, \emph{Quelques applications du th\'eor\`eme de densit\'e de
  {C}hebotarev}, Inst. Hautes \'Etudes Sci. Publ. Math. (1981), no.~54,
  323--401. \MR{644559 (83k:12011)}

\bibitem{sutherland}
Andrew~V. Sutherland, \emph{A local-global principle for rational isogenies of
  prime degree}, J. Th\'eor. Nombres Bordeaux \textbf{24} (2012), no.~2,
  475--485. \MR{2950703}

\bibitem{sz}
Andrew~V. Sutherland and David Zywina, \emph{Modular curves of prime-power
  level with infinitely many rational points}, Algebra Number Theory
  \textbf{11} (2017).

\bibitem{ehom}
Isabel Vogt, \emph{Abelian varieties isogenous to a power of an elliptic curve
  over a {G}alois extension}, J. Th\'{e}or. Nombres Bordeaux \textbf{31}
  (2019), no.~1, 205--213. \MR{3994726}

\end{thebibliography}
\end{document}